\documentclass[11pt, reqno]{amsart}

\usepackage{amssymb}
\usepackage{amsmath}
\usepackage{amsthm}
\usepackage{physics}
\usepackage{xcolor}
\usepackage{graphicx}
\usepackage{subcaption}

\usepackage[linesnumbered,lined,boxed,commentsnumbered]{algorithm2e}

\usepackage{hyperref}
\usepackage{comment}
\usepackage{subcaption}
\usepackage{adjustbox}
\usepackage[normalem]{ulem} 

\def\R{\mathbb{R}}
\def\Z{\mathbb{Z}}

\def\M{\mathcal{M}}

\newcommand{\dsum}{\displaystyle\sum}

\def\x{\bold{x}}

\newtheorem{theorem}{Theorem}
\newtheorem{lemma}[theorem]{Lemma} 
 
\newtheorem{corollary}{Corollary}

\newtheorem{definition}{Definition}[section]
\newtheorem{example}{Example}[section] 
 
\newtheorem{remark}{Remark}[section] 
\usepackage{geometry}
\usepackage{natbib}

\let\origmaketitle\maketitle
\def\maketitle{
  \begingroup
  \def\uppercasenonmath##1{} 
  \let\MakeUppercase\relax 
  \origmaketitle
  \endgroup
}

\def\R{\mathbb{R}}

\def\Z{\mathbb{Z}}

\def\MM{\mathbb{M}}

\renewcommand{\x}{\ensuremath{{\bm{x}}\xspace}}
\newcommand{\s}{\ensuremath{{\bm{s}}\xspace}}
\newcommand{\y}{\ensuremath{{\bm{y}}\xspace}}
\newcommand{\z}{\ensuremath{{\bm{z}}\xspace}}
\newcommand{{\ww}}{\ensuremath{{\bm{w}}\xspace}}

\newcommand{\e}{\ensuremath{{\bm{e}}\xspace}}
\renewcommand{\d}{\ensuremath{{\bm{d}}\xspace}}
\renewcommand{\t}{\ensuremath{{\bm{t}}\xspace}}
\usepackage{bm} 
\newcommand{{\vv}}{\bm{v}} 
\newcommand{\cc}{{\bm{c}}} 
\renewcommand{\c}{{\bm{c}}}

\newcommand{\uu}{\ensuremath{{\bm{u}}\xspace}}
\newcommand{\bb}{\ensuremath{{\bm{b}}\xspace}}
\renewcommand{\a}{\ensuremath{{\bm{a}}\xspace}}
\newcommand{\0}{\ensuremath{{\bm{0}}\xspace}}

\newcommand{\ba}{\ensuremath{\bm{\alpha}\xspace}}
\newcommand{\be}{\ensuremath{\bm{\beta}\xspace}}
\newcommand{\ga}{\ensuremath{\bm{\gamma}\xspace}}

\newcommand{\XI}{\ensuremath{\bm{\xi}\xspace}}
\newcommand{\Eta}{\ensuremath{\bm{\eta}\xspace}}

\newcommand\MinMG{\ensuremath{\mathrm{MinMG}\xspace}}

\newcommand\MaxMG{\ensuremath{\mathrm{MaxMG}\xspace}}
\newcommand{\A}{\mathcal{A}}
\newcommand\B{\ensuremath{\mathcal{B}\xspace}}
\newcommand\C{\ensuremath{\mathcal{C}\xspace}}
\renewcommand\M{\ensuremath{\mathcal{M}\xspace}}

\newcommand\V{\ensuremath{\mathcal{V}\xspace}}

\newcommand\GPC{\ensuremath{\mathrm{GPC}\xspace}}
\newcommand\SOC{\ensuremath{\mathrm{SOC}\xspace}}
\newcommand\SONC{\ensuremath{\mathrm{SONC}\xspace}}
\newcommand\SOS{\ensuremath{\mathrm{SOS}\xspace}}
\newcommand\conv{\ensuremath{\mathrm{Conv}\xspace}}

\newcommand\adj{\ensuremath{\mathrm{Adj}\xspace}}

\renewcommand\int{\ensuremath{\mathrm{int}\xspace}}

\newcommand{\dprod}{\displaystyle\prod}
\newcommand{\dmax}{\displaystyle\max}
\newcommand{\dmin}{\displaystyle\min}

\newtheorem{properties}{Properties}[section]%

\usepackage{hyperref}
\usepackage{tikz}
\usepackage{pgfplots}
\pgfplotsset{compat=newest}
\usetikzlibrary{arrows.meta,calc,decorations.markings,math}
\usepackage{schemata}%
\usepackage{multirow}

\title[Optimal Mediated Graphs]{{\Large Optimal mediated graphs}\\\medskip
{\large The role of combinatorics in conic optimization}}

\author[V. Blanco \MakeLowercase{and} M. Mart\'inez-Ant\'in]{{\large V\'ictor Blanco$^{\dagger,\star}$  and Miguel Mart\'inez-Ant\'on$^\dagger$}\medskip\\
$^\dagger$Institute of Mathematics (IMAG), Universidad de Granada\\
$^\star$ Dpt. Quant. Methods for Economics \& Business, Universidad de Granada}

\begin{document}

\maketitle
\vspace{-2em}
\begin{center}
    {\small \texttt{vblanco@ugr.es}, \texttt{mmanton@ugr.es}}
\end{center}


\begin{abstract}
In this paper, we provide a unified definition of mediated graph, a combinatorial structure with multiple applications in mathematical optimization. We study some geometric and algebraic properties of this family of graphs and analyze extremal mediated graphs under the partial order induced by the cardinalty of their vertex sets. We derive mixed integer linear formulations to compute these challenging graphs and show that these structures are crucial in different fields, such as sum of squares decomposition of polynomials and second-order cone representations of convex cones, with a direct impact on conic optimization. We report the results of an extensive battery of experiments to show the validity of our approaches.
\end{abstract}




\section*{Introduction}

Conic optimization plays a fundamental role in combinatorial optimization and graph theory, where conic structures have been used to formulate stronger relaxations~\citep{bie2006fast}, derive efficient algorithms for challenging combinatorial problems~\citep{alizadeh1995interior}, or detect graph properties or invariants, as those induced by semidefinite programming (SDP) for problems such as graph partitioning~\citep{orecchia2011towards}, the maximum stable set problem~\citep{gaar2022sdp}, and the max-cut problem~\citep{gaar2020computational,goemans1995improved}. Beyond SDPs, second-order cone programming (SOCP) and $p$-order cone formulations ($p$OCP) have also been leveraged to strengthen or find good quality bounds for clasical combinatorial optimization problems or to model robust versions of these problems~\citep[see, e.g.,][among many others]{burer2009p,muramatsu2003new,bental2001,buchheim2018robust,liu2022robust}. Conic duality has been instrumental in identifying hidden combinatorial structures, such as exploiting Laplacian eigenvalues in spectral graph theory to design more efficient clustering and community detection algorithms~\citep{qing2020dual}. 

Needless to say that conic optimization provides by itself a fundamental framework for addressing a broad class of convex optimization problems, and that is particularly useful, through convenient relaxations via the Moment-SOS hierarchy~\citep{lasserre2001global}, to optimize multivariate polynomials restricted on feasible regions defined by semialgebraic sets. At the core of this methodology lies the ability to certify polynomial nonnegativity using SOS decompositions, which translates nonnegativity conditions into semidefinite constraints. This approach is particularly powerful in several applications, as optimal control, probability, or dynamical systems where ensuring stability, safety, and optimality often relies on verifying polynomial inequalities~\citep{pauwels2017positivity,parrilo2003semidefinitea,parrilo2003semidefiniteb,henrion2020moment}. Within this framework, two fundamental problems arise: (1) the decomposition of nonnegative polynomials as sum of squares polynomial, since it implies a semidefinite programming reformulation of a polynomial optimization problem; and (2) the representation of general $p$-order conic problems as second-order conic problems that are efficiently handled by the off-the-shelf solvers.

On the other hand, graph theory is fundamental in understanding the structure and interrelationships within complex systems and has been proven to be a very powerful tool to analyze social or telecommunication networks~\citep[see, e.g.,][]{barnes1969graph,ramirez2020multilayer}. 
Nevertheless, graph theory holds fundamental importance beyond its practical applications, as its ability to translate abstract mathematical ideas into concrete visual forms  provides a rich framework for exploring deep mathematical concepts related to structure, symmetry, and connectivity, offering insights in different more theoretical disciplines. For instance, graph-theoretic approaches are used to explore knot theory~\citep{murasugi1989invariants}, the study of surfaces~\citep{boykov2003computing} or probabilistic methods~\citep{erdos1960evolution,gilbert1959random}.

In this paper, we provide a new link between these two worlds, conic optimization and graph theory, by introducing and analyzing a new family of geometric graphs, which despite its impact the computation of explicit sum of squares decomposition of nonnegative polynomials or SOCP representation of general $p$-order and power cone problems, has not been previously formally introduced, the so-called family of \emph{mediated graphs}.

We formally analyze the family of mediated graphs, that will be defined as geometric directed graphs embedded in abelian groups, where almost all its vertices (parents) are the midpoint of two other vertices (children) in the graph, and an arc links each parent with its children. The underlying structure of these graphs is closely related to the concept of mediated set that was introduced by \cite{reznick1989forms} as a tool to provide a necessary and sufficient condition for an \emph{agiform} (nonnegative homogeneous polynomial derived from ``arithmetic-geometric mean inequality'') to be decomposed as a sum of squares of polynomials. Since then, the notion of mediated set has been modified to be adapted to different problems (see, e.g., \cite{hartzer2022initial,magron2023sonc,wang2024weighted}). \cite{magron2023sonc} also exploit this structure to provide a second-order cone ($\SOC$) representation of sums of nonnegative circuits cones ($\SONC$). \cite{dressler2017positivstellensatz} and \cite{wang2022nonnegative} propose the use of mediated sets to represent convex semialgebraic sets using $\SONC$. Sum of squares ($\SOS$) decompositions were addressed by \cite{reznick1989forms, iliman2016amoebas, iliman2016lower} using also mediated sets. Recently, in \citep{wang2024weighted} these sets have been applied to provide optimal $\SOC$ representations of weighted geometric mean inequalities. The usefulness of the graph structure of a mediated set was already observed by \cite{blanco2024minimal} for the same problem, where the authors provide explicit minimal $\SOC$ extended representations of generalized power cones ($\GPC$). It allows to exactly and efficiently formulate a $\GPC$ Programming (GPCP) problem as a $\SOC$ Programming (SOCP) problem with a direct impact in solving optimization problems involving these cones. Another recent contributions on mediated sets can be found in~\citep{powers2021note}.

The remainder of this paper is organized as follows. In Section \ref{sec:prelim}, we introduce the notation used throughout the paper, formally define the concept of mediated graphs, and establish some fundamental geometric and algebraic properties derived from their definition. Section \ref{sec:optmg} is dedicated to studying optimal (extremal) graphs within the class of mediated graphs under a partial order that we define, based on the cardinality of the vertex set. The identification and computation of these optimal graphs form the core of this work. We motivate the need for their efficient computation by highlighting their connection to the existence of SOS decompositions of circuit polynomials and minimal SOC representations of more complex conic structures. In Sections \ref{sec:minimal} and \ref{sec:maximal}, we introduce and analyze minimal and maximal mediated graphs, respectively. These sections present the main contributions of this work: the development of integer linear optimization (ILO) formulations to compute these graphs. We further illustrate the application of these structures in three different type of problems of interest: representation of generalized power cones, SOS decomposition of sum of nonnegative circuits (SONC), and the analysis of the intersection of the SOS and SONC cones. In Section \ref{sec:exp}, we report the results of an extensive set of experiments demonstrating the validity and practical utility of our approaches. For maximal mediated graphs, we compare our optimization-based method with the enumerative approach proposed in \citep{hartzer2022initial}. For minimal mediated graphs, we analyze the performance of different domain formulations for which we derive integer linear optimization models. Finally, we draw some conclusions and indicate some potential future research on the topic.

\section{Mediated Graphs}\label{sec:prelim}

In this section, we introduce the notion of mediated graph and derive some interesting geometric and algebraic properties. 

\subsection{Notation}

Let $G:=(V,A)$ be a directed graph, where $V$ is the vertex set and $A\subset V\times V$ is the arc set of $G$. For each ${\vv} \in V$, we denote by $\delta^+({\vv}):=\{{\ww} \in V: ({\vv},{\ww}) \in A\}$
the set of its outgoing arcs, and $\deg^+({\vv}):=|\delta^+({\vv})|$ is the outdegree (here, $|X|$ denotes the cardinal of the finite set $X$). Analogously, $\delta^-({\vv}):=\{{\ww} \in V: ({\ww},{\vv}) \in A\}$ denotes the set of ingoing arcs of ${\vv}$, and $\deg^-({\vv}):=|\delta^-({\vv})|$ its indegree. Any directed graph, $G=(V,A)$, has associated three matrices that allow to characterize the graph, namely the adjacency, the degree, and the laplacian matrix. The \emph{adjacency matrix} of $G$ is defined as $\adj(G):=(a_{{\vv} {\ww}})_{{\vv},{\ww}\in V},$ where $a_{{\vv}{\ww}}=1$ if $({\vv},{\ww})\in A$, and zero otherwise.  The \emph{degree matrix} of $G$ is the diagonal matrix of order $|V|\times |V|$ defined as $D^+(G):=(d_{{\vv} {\ww}})_{{\vv},{\ww}\in V},$ where $d_{{\vv} {\vv}}=\deg^+({\vv})$. Finally The \emph{laplacian matrix} of $G$ is defined as $L^+(G):=D^+(G)-\adj(G).$ Given a subset $U\subset V$, a matrix $M_U$ corresponds with the submatrix of $M$ formed by the rows and columns indexed in $U$.

Let $\R^d$ be the $d$-dimensional real vector space, we denote the vectors ${\vv}=(v_1,\ldots, v_d)\in \R^d$ by bold characters. Given a subset $\MM \subseteq \R^d$, we say that $G=(V,A)$ is a $\MM$-geometric digraph if its vertex set, $V$, is embedded in $\MM$, i.e., $V \subset \MM$. For a finite set $\A\subset \MM$, we denote by $\conv(\A)$ the convex hull of $\A$, and by $\conv(\A)^\circ$ its interior.  

We start by introducing the main definition in this work.
\begin{definition}[Mediated Graph]\label{def:mg}
    Let $\A \subset \MM$ be a finite set. A $\MM$-geometric digraph $G=(V,A)$ is said a {\bf $\A$-mediated graph} if $\A\subset V$, $\delta^+({\vv}) \neq \emptyset$, and if ${\ww} \in \delta^+({\vv})$, then $2{\vv}-{\ww} \in \delta^+({\vv})$, for all ${\vv}\in V\backslash \A$.
\end{definition}

The family of $\A$-mediated graphs on $\MM$ will be denoted as $\M_\A^{\MM}$. The mention of the domain $\MM$ will be omitted in the notation, unless it is necessary to avoid confusion.

The definition of mediated graph implies that every vertex in $V$, except those in $\A$, is the midpoint of two of its out-adjacent vertices.

\begin{remark}\label{rem:1}
    Note that every nonempty mediated graph $G=(V,A)$ can be reduced to a graph $G'=(V,A')$, whose outdegree $\deg^+({\vv})=2$ for all ${\vv}\in V\backslash \A$ and $\deg^+(\a)=0$ for all $\a \in \A$ which is also a mediated graph. Thus, hereinafter, we assume without loss of generality, that, in every $\A$-mediated graph, each vertex has outdegree $2$, except for those in $\A$ which have outdegree $0$. Moreover, we assume $\A\neq \emptyset$, since the only $\emptyset$-mediated graph is the empty graph.
\end{remark}

In the following example, we illustrate the geometrical shapes of some mediated graphs.
\begin{example}
Let $\A=\{(0,0),(7,0),(0,7)\}$ (red dots) in Figure \ref{fig:ejemplomg} we show two different $\A$-mediated graphs in two different domains.  In the left picture we show $\A$-mediated graph constructed on $\MM=\R^2$. Note that, apart from the points in $\A$, four more vertices appear in the graph (blue dots), namely $(0.5,0.5), (1,1), (2,2)$, and $(3.5,3.5)$, all of them midpoints of two other vertices in the graph. The arrows indicate the arcs in the graph, each of them directed from the vertex to the two other vertices for which it is a midpoint. In the right picture we show an example of $\A$-mediated graph with domain $\MM=\Z^2$. In this case, the graph consists of $10$ vertices ($3$ of them those in $\A$), but the coordinates of the points are now integer numbers. In this case, the vertices not in $\A$ are: 
  $(1, 0)$, $(1, 1)$, $(1, 2)$, $(2, 0)$, $(2, 4)$, $(4, 0)$, and $(4, 1)$.

\begin{figure}[h!]
\begin{center}
\adjustbox{width=0.45\textwidth}{\begin{tikzpicture}[scale=0.35]
\draw [step=1.0, gray, very thin] (0,0) grid (7.0,7.0);
\node (A1) at (0.0,7.0) [circle, draw, red, inner sep=0.8, fill=red] {};
\node (A2) at (0.5,0.5) [circle, draw, blue, inner sep=0.8, fill=blue] {};
\node (A3) at (0.0,0.0) [circle, draw, red, inner sep=0.8, fill=red] {};
\node (A4) at (1.0,1.0) [circle, draw, blue, inner sep=0.8, fill=blue] {};
\node (A5) at (7.0,0.0) [circle, draw, red, inner sep=0.8, fill=red] {};
\node (A6) at (3.5,3.5) [circle, draw, blue, inner sep=0.8, fill=blue] {};
\node (A7) at (2.0,2.0) [circle, draw, blue, inner sep=0.8, fill=blue] {};

\draw[-{Stealth[scale=0.4]}, line width=0.2] (A7)--(A6);
\draw[-{Stealth[scale=0.4]}, line width=0.2] (A6)--(A1);
\draw[-{Stealth[scale=0.4]}, line width=0.2] (A7)--(A2);
\draw[-{Stealth[scale=0.4]}, line width=0.2] (A2)--(A3);
\draw[-{Stealth[scale=0.4]}, line width=0.2] (A4)--(A7);
\draw[-{Stealth[scale=0.4]}, line width=0.2] (A6)--(A5);
\draw[-{Stealth[scale=0.4]}, line width=0.2] (A2)--(A4);
\draw[-{Stealth[scale=0.4]}, line width=0.2] (A4)--(A3);
\end{tikzpicture}}~\hspace*{0.3cm}\adjustbox{width=0.45\textwidth}{\begin{tikzpicture}[scale=0.35]
\draw [step=1.0, gray, very thin] (0,0) grid (7,7);
\node (A1) at (0,7) [circle, draw, red, inner sep=0.8, fill=red] {};
\node (A2) at (2,4) [circle, draw, blue, inner sep=0.8, fill=blue] {};
\node (A3) at (4,0) [circle, draw, blue, inner sep=0.8, fill=blue] {};
\node (A4) at (1,2) [circle, draw, blue, inner sep=0.8, fill=blue] {};
\node (A5) at (0,0) [circle, draw, red, inner sep=0.8, fill=red] {};
\node (A6) at (1,1) [circle, draw, blue, inner sep=0.8, fill=blue] {};
\node (A7) at (7,0) [circle, draw, red, inner sep=0.8, fill=red] {};
\node (A8) at (2,0) [circle, draw, blue, inner sep=0.8, fill=blue] {};
\node (A9) at (1,0) [circle, draw, blue, inner sep=0.8, fill=blue] {};
\node (A10) at (4,1) [circle, draw, blue, inner sep=0.8, fill=blue] {};
\draw[-{Stealth[scale=0.4]}, line width=0.2] (A4)--(A5);
\draw[-{Stealth[scale=0.4]}, line width=0.2] (A10)--(A7);
\draw[-{Stealth[scale=0.4]}, line width=0.2] (A9)--(A8);
\draw[-{Stealth[scale=0.4]}, line width=0.2] (A8)--(A5);
\draw[-{Stealth[scale=0.4]}, line width=0.2] (A6)--(A9);
\draw[-{Stealth[scale=0.4]}, line width=0.2] (A6)--(A4);
\draw[-{Stealth[scale=0.4]}, line width=0.2] (A2)--(A1);
\draw[-{Stealth[scale=0.4]}, line width=0.2] (A3)--(A9);
\draw[-{Stealth[scale=0.4]}, line width=0.2] (A8)--(A3);
\draw[-{Stealth[scale=0.4]}, line width=0.2] (A9)--(A5);
\draw[-{Stealth[scale=0.4]}, line width=0.2] (A3)--(A7);
\draw[-{Stealth[scale=0.4]}, line width=0.2] (A2)--(A10);
\draw[-{Stealth[scale=0.4]}, line width=0.2] (A4)--(A2);
\draw[-{Stealth[scale=0.4]}, line width=0.2] (A10)--(A4);
\end{tikzpicture}}
    \caption{Two versions of $\A$-mediated graphs but with different vertex domains. $\R^2$ (left) and $\Z^2$ (right).\label{fig:ejemplomg}}
    \end{center}
\end{figure}
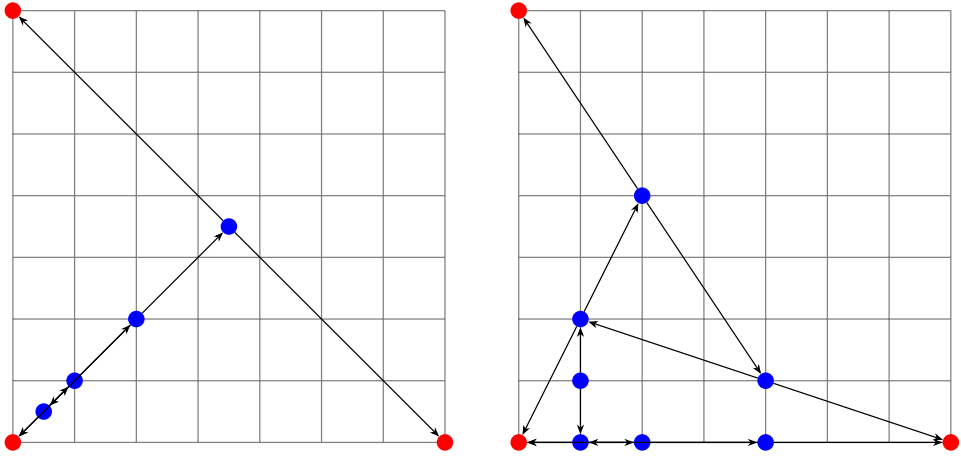

\end{example}

Although mediated graphs have been proven to be useful in different fields, in this paper we provide a general framework for them, as well as explore some of their properties. 
\subsection{Properties}

First, we derive some geometric and algebraic properties of a mediated graph, whose proof is straightforward, but allows us to understand the closeness properties of the family $\M_\A$.

\begin{properties}\label{prop:prop}
Let $\A\subset\MM$ be a finite set of points in $\MM$, and $G=(V,A) \in \M_\A$. Then, the following properties are verified:

\begin{enumerate}
    \item $G \in \M_A^{\MM'}$ for any $\MM \subseteq \MM'$.\label{prop:1} 
    \item $G \in \M_{\A'}$ for all $\A\subseteq \A'\subset \MM$ (without considering the isolated vertices in $\A'\backslash\A$).\label{prop:2}
    \item $\M_\emptyset = \{(\emptyset,\emptyset)\}$.\label{prop:4}
    \item If $\A'\subset \MM$ and $(V',A') \in \M_{\A'}$, then $(V\cup V', A\cup A')\in \M_{\A\cup\A'}$.\label{prop:union}
     \item $V \subset \conv(\A)$.\label{prop:chull}
     \item If $\A$ is affinely independent and $V\cap \conv(\A)^\circ\neq \emptyset$, then $\deg^{-}(\a)\geq 1$ for all $\a\in\A$.\label{prop:last}
     \item The submatrix $L^+(G)_{V\backslash \A}$ of the laplacian matrix of $G$ is invertible, and the entries of its inverse are nonnegative.\label{prop:laplacian}
\end{enumerate}
\end{properties}

\begin{proof}
 Since \ref{prop:1}-\ref{prop:last} are straightforward, we will only prove \ref{prop:laplacian}. It is just needed to see if the matrix satisfies the hypothesis of ~\cite[Lemma 4.3]{reznick1989forms}. Since, for all ${\vv}\in V\backslash \A$, $\deg^+({\vv})=2$, thus the entries of the main diagonal of the matrix are equal to two; and it is clear that the entries out of the main diagonal only take values in $\{0,-1\}$. Each row is in the shape $(c_{{\vv}{\ww}})_{{\ww}\in V\backslash \A}$ with ${\vv}\in V\backslash \A$, therefore it has as $-1$'s as children of ${\vv}$ in $V\backslash \A$ so at most two. Eventually, assume that if there exists a principal submatrix in which each row has exactly two $-1$'s. This implies that there exists $U\subseteq V\backslash \A$ such that for all ${\vv}\in U$ its two children belong to $U$. So, the graph $(U, A\cap (U\times U))$ is an $\emptyset$-mediated graph, then $U=\emptyset$ by statement \ref{prop:4}.
\end{proof}

In the family of mediated graphs $\M_\A$, we will analyze those graphs whose vertex set contains a finite set $\B \subset \MM$. These subsets play a main role in this work.

\begin{equation}\label{eq:subsets}
\M_\A(\B) := \left\{(V,A) \in \M_\A: \B \subset V\right\}.
\end{equation}

When $\B$ be a singleton, i.e. $\B =\{\bb\}$ for some $\bb \in \MM$, $\M_\A(\bb) := \M_\A(\{\bb\})$. 

In what follows we analyze under which conditions the family $\M_\A(\B)$ is not empty.

\begin{definition}[Binary Support]
    Let $n$ be a nonnegative integer number. We denote by $\Omega(n)$ the set of strictly positive coefficients in the binary decomposition of $n$, namely, if $n = \sum_{i=0}^{\lfloor \log_2(n)\rfloor} b_i 2^i$, then
$$
\Omega(n) = \left\{i \in \{0, \ldots, \lfloor \log_2(n)\rfloor\}: b_i=1\right\}.
$$
\end{definition}

We adapt \cite[Corollary 28]{wang2024weighted} for mediated graphs to provide a first nonemptyness sufficient condition for $\M_\A(\bb)$.
\begin{lemma}\label{lemma:wang}
    Let $\A \subset \R^d$ be an affine independent set and $\bb\in \conv(\A)^\circ$. Then, there exists nonempty graph $G=(V,A)$ in $\M_{\A}(\bb)$ with
$$
|V|=\dsum_{\a\in \A}|\Omega(s_\a)|+|\A|-1.
$$
where $\s=(s_\a)_{\a\in\A}\in (\Z_+^*)^{|\A|}$, $\gcd(\s)=1$, and $\frac{1}{\|\s\|_1}\s$ are the barycentric coordinates of $\bb$ with respect to $\A$.
\end{lemma}

In the following result, we generalize the previous result for any set $\A$ and any set $\B$ (not necessarily singleton). Notice that for every point $\bb\in\conv(\A)$, there is at least one subset $\C\subseteq \A$ affinely independent (a.i.), strictly containing the point $\bb$. Hence, we can define 
\begin{multline}
    \A(\bb):={\rm arg} \dmin_{\C\subseteq \A \text{ a.i.}}\Bigg\{\dsum_{\cc\in \C}|\Omega(s_\cc)| + |\C|: \\\s=(s_\cc)_{\cc\in \C}\in (\Z_{
    >0})^{|\C|}, \; \gcd(\s)=1, \; \bb=\frac{1}{\|\s\|_1}\dsum_{\cc\in \C}s_\cc \cc \Bigg\}.
\end{multline}

We denote by $\s^\bb:=(s_\a^\bb)_{\a\in \A(\bb)}\in (\Z_+^*)^{|\A(\bb)|}$, where $\gcd(\s^\bb)=1$ and $\frac{1}{\|\s^\bb\|_1}\s^\bb$ are the barycentric coordinates of $\bb$ with respect to $\A(\bb)$. 

\begin{theorem}\label{theo:ub}
Let $\A,\B\subset \R^d$ be finite subsets and $\B\subset \conv(\A)$. Then, there exists a nonempty graph $G = (V,A) \in \M_\A(\B)$ such that:
\begin{equation}\label{eq:ub}
|V| \leq \nu_{\A}(\B):=\dsum_{\bb \in \B}\left[ \dsum_{\a\in\A(\bb)}|\Omega(s_\a^\bb)|+ |\A(\bb)|\right]-|\B|.
\end{equation}
\end{theorem}
\begin{proof}
By Lemma \ref{lemma:wang}, there exists $G_{\A(\bb)}=(V_{\A(\bb)},A_{\A(\bb)})$ in $\M_{\A(\bb)}(\bb)\subseteq \M_\A(\bb)$ (the inclusion is obtained from Properties \ref{prop:prop}.\ref{prop:2}) with
$$
|V_{\A(\bb)}|=\dsum_{\a\in \A(\bb)}|\Omega(s_\a^\bb)|+|\A(\bb)|-1.
$$
By Properties \ref{prop:prop}.\ref{prop:union}, there must be a graph $G=(V, A)\in \M_\A(\B)$ where $V=\displaystyle\bigcup_{\bb\in \B}V_{\A(\bb)}$, thus $|V|\leq \dsum_{\bb\in \B} |V_{\A(\bb)}|$.
\end{proof}

Note that the nonemptyness of $\M_\A(\B)$ is not assured in general, even if $\B\subset \conv(\A)\cap \MM$. In the following result we show a counterexample of this situation.

\begin{example}
    Let us consider $\MM=(2\Z)^2$, $\A=\{(0,0),(4,2),(2,4)\}$ and $\bb=(2,2)\in \conv(\A)\cap (2\Z)^2$. Note that there is choice to construct a $\A$-mediated graph that contains $\bb$, and then $\M_\A(\bb)=\emptyset$.
\end{example}

 In view of the possible nonexistence of mediated graphs, \cite{powers2021note} provide a lower bound for the \emph{dilation} constant $k\in \Z_+$ such that $\M_{k\A}(\B)$ is nonempty for all  $\A\subset \MM$, where $k\A:=\{k\a\in (2\Z)^d: \a\in\A\}$ and $\B\subset \conv(k\A)\cap \Z^d$ in the case $\MM=(2\Z)^d$.

\begin{theorem}[\cite{powers2021note}]
    Let $\A\subset (2\Z)^d$ be an affine independent set. Then, for every integer $k\geq \{2, |\A|-2\}$, there exists a $(k\A)$-mediated graph, $G=(\conv(k\A)\cap \Z^d,A)$ such that $\delta^+({\vv})\subset (2\Z)^d$ for every vertex ${\vv}\in \conv(k\A)\cap \Z^d$.
\end{theorem}

\section{Optimal Mediated Graphs}\label{sec:optmg}

The family of $\A$-mediated graphs can be endowed with a partial order induced by the cardinality of their vertex sets. Specifically, given $G=(V,A), G'=(V',A')\in \M_\A$:

\begin{equation}
    G\preceq G' \text{ if either }|V|<|V'| \text{ or }G=G'.
\end{equation} 

Based on this partial order, the family of $\A$-mediated graphs for a given finite set $\A \subset \MM$ has two distinguished elements that arise when minimizing and maximizing the number of vertices in the graph.  Since the enumeration of this family of graphs can be, in general, cumbersome, we develop approaches to construct these \emph{optimal} graphs using mathematical optimization tools, which, as already announced, is the main contribution of this paper. The study and construction of these so-called minimal and maximal mediated graphs is motivated by their implications in different problems in convex algebraic geometry, and then, in its application to efficiently represent useful convex optimization problems.

We separately analyze minimal and maximal mediated graphs, because of different reasons. On the one hand, Minimal Mediated Graphs (MinMG, for short) have already been proven to be useful structures to derive the most efficient and simple SOC or SOS representations of other more complex convex sets, with a high impact in practical convex optimization problems~\cite[see, e.g.,][]{blanco2014revisiting,blanco2024minimal,wang2024weighted}. Furthermore, MinMG can be analyzed in \emph{discrete} and \emph{continuous} domains referring to the cardinality of $\conv(\A)\cap \MM$, finite and infinite, respectively. On the other hand, the known applications of Maximal Mediated Graphs (MaxMG, for short)  are related to the existence of SOS decomposition of circuit polynomials \citep{hartzer2022initial,reznick1989forms}. Clearly, the construction of MaxMG is only possible in discrete domains. 

In this section, we analyze some properties of MinMG and MaxMG and propose different approaches for their construction.

\subsection{Minimal Mediated Graphs}\label{sec:minimal}

Note that the unique minimal mediated graph of $\M_\A$ is the edgeless graph with vertex set $\A$. Thus, we analyze here MinMG restricted to the family $\M_\A(\B)$  for a nonempty set $\B \subset \conv(\A)\cap\MM$ (see \eqref{eq:subsets}). In what follows, we formally introduce this subgraph.

\begin{definition}
    Let $\A, \B \subset \MM$ be finite sets. We say that $G=(V,A) \in \M_\A(\B)$, is a Minimal $\A$-Mediated Graph for $\B$ if there not exists $G'=(V',A') \in \M_\A(\B)$ such that $|V'|<|V|$.
\end{definition}

We denote by $\MinMG_\A(\B)$ the family of minimal $\A$-mediated graphs for $\B$. Notice that a minimal mediated graph is a minimal element in the poset $(\M_\A(\B),\preceq)$.


Some properties can be derived for this family of mediated graphs 

\begin{theorem}\label{theo:connected}
   Let $\A\subset \MM$ and $\B \subset \MM$. Then, any minimal mediated graph $G\in \M_\A(\B)$ has at most $|\B|$ weak connected components.
\end{theorem}

\begin{proof}
Let us first analyze the case when $\B=\{\bb\}$. In case $\M_\A(\bb)=\emptyset$ is clear. We prove the lemma by contraposition for nonempty $\M_\A(\bb)$. Suppose the graph $G=(V,A)\in \M_\A(\bb)$ is not weakly connected. Thus, one can construct the graph, $G'=(V',A')$ where: 
\begin{gather*}
    V':=\{{\vv}\in V: \text{ there exists a path from }\bb\text{ to }{\vv}\}\\
    A':=A\cap(V'\times V').
\end{gather*}
Note that $G'$ is a $\A$-mediated graph with $\bb\in V'$. Given ${\vv}\in V'\backslash \A$, by definition, there exists a path in $G$ from $\bb$ to ${\vv}$. If ${\ww}$ and $2{\vv}-{\ww}$ are the children of ${\vv}$ in $G$, then ${\ww}, {\vv}-{\ww} \in V'$, and then $({\vv},{\ww}), ({\vv},2{\vv}-{\ww}) \in A'$. Since, clearly, $\bb \in V'$, $G'\in \M_\A(\bb)$. As $G$ is not weakly connected, $V'\subsetneq V$, then $|V'|<|V|$, contradicting the minimality of $G$.

    Let us now analyze the general case. If $G=(V,A)\in\M_\A(\B)$ which is not weakly connected, it is enough to consider the graph $G''=(V'',A'')$ defined by
    
    \begin{gather}\label{eq:path}
        V'':=\{{\vv}\in V: \text{ there exists a path from }\B\text{ to }{\vv}\},\\
        A'':=A\cap(V''\times V'')\nonumber
    \end{gather}
    which, analogously to the singleton case, is $\M_\A(\bb)$, with $|V''|<|V|$, contradicting the minimality of $G$.
\end{proof}

The following result is straightforward from the construction in the previous result, and provides us information about the indegrees of the vertices in a minimal mediated graph.
\begin{corollary}
    If $G\in \M_\A(\B)$ is a minimal mediated graph, then $\deg^-({\vv})\geq 1$ for all ${\vv}\in V\backslash \B$.
\end{corollary}

The following result is a straightforward consequence of Theorem \ref{theo:ub} giving us an upper bound for the cardinality of a minimal mediated graph.

\begin{corollary}\label{coro:ub}
Let $\A, \B \subset \R^d$ and $G=(V,A)\in \MinMG_\A(\B)$. Then, $$|V|\leq \nu_{\A}(\B).$$
\end{corollary}



Computing $\MinMG_\A(\B)$ requires determining both the vertices in $\conv(\A)$ and the arcs that verify the \emph{mediated} condition at minimum cardinality. For particularly structured sets $\A$, and singleton sets $\B$, \cite{wang2024weighted} propose a brute force algorithm and different heuristic approaches to compute $\MinMG_\A(\B)$ focused on deriving minimal SOCP representations of weighted geometric mean inequalities. \cite{blanco2024minimal} provide a mixed integer linear programming model (MILP) to compute this set. In both cases, $\MM=\R^d$ for a given dimension $d\geq 1$. This framework will be named in this work as the \emph{continuous domain} minimal mediated graph, since the vertices are to be found in the continuous region $\conv(\A)$. On the other hand, in case $\MM$ is a \emph{lattice} (as $\Z^d$ or $(2\Z)^d$), we call this framework the \emph{{discrete domain}} minimal mediated set. The main difference stems in that the search of the vertices and links in the continuous domain case is to be done among infinitely many vectors, whereas in the discrete domain case, the search is done in a finite set of points. By analogy with facility location theory, this is exactly the same difference between the family of continuous location problems and discrete location problems.

In what follows we analyze minimal mediated sets for these two families of domains and derive mathematical optimization models that allows to compute the minimal mediated graphs, avoiding, in general the enumeration of the potential minimal subgraphs.

\subsubsection*{Minimal Mediated Graphs in $\R^d$}

The problem of deriving mathematical optimization models for the continuous case was already addressed by \cite{blanco2024minimal} for the vertex sets of dilations of the standard simplex and a single interior lattice point. Below we state the result for the general continuous domain case, which is a straightforward extension of the formulation provided in the mentioned paper.

\begin{theorem}\label{th:continuous}
Let $\A\subset \R^d$ be a finite set of points, and $\B \subset \conv(\A)$. The following mixed integer linear programming model allows to compute $\MinMG_\A(\B)$ for $\MM=\R^d$.

\begin{align}\label{eq:cminmg}\tag{C-MinMGP}
&& \nonumber\\
\text{Minimize } & \; |\A| + |\B|  + \sum_{{\vv}\in \V\backslash(\A\cup \B)} z_{\vv} & \label{c:obj}\\
\text{subject to }
& \; y_{{\vv}{\ww}} \leq z_{\vv},  & \forall {\vv} \in \V\backslash \A, {\ww}\in \V,\label{c:2}\\
& \; y_{{\vv}{\ww}} \leq z_{\ww},  & \forall {\vv} \in \V\backslash \A, {\ww}\in \V, \label{c:3}\\
& \; 2\x_{\vv} \geq \x_{\ww} + \x_{\uu}  - \Delta (2- y_{{\vv}{\ww}}-y_{{\vv}{\uu}}), & \forall {\vv} \in \V\backslash \A, {\ww}\neq {\uu} \in \V, \label{c:4}\\
& \; 2\x_{\vv} \leq \x_{\ww} + \x_{\uu} + \Delta (2- y_{{\vv}{\ww}}-y_{{\vv}{\uu}}), & \forall {\vv} \in \V\backslash\A, {\ww}\neq {\uu} \in \V,\label{c:5}\\
 & \; \|\x_{\vv} - \x_{\ww}\|_1 \geq \varepsilon (z_{\vv}+z_{\ww}-1), & \forall {\vv}\neq{\ww} \in \V,\label{c:6}\\
& \; \dsum_{{\ww} \in \V} y_{{\vv}{\ww}} = 2 z_{\vv}, & \forall {\vv} \in \V\backslash \A,\label{c:1}\\
 & \; \x_{{\vv}} = {\vv}, \; z_{{\vv}}=1 & \forall {\vv} \in \A\cup\B, \label{c:7}\\
& \; y_{{\vv}{\ww}} \in \{0,1\}, & \forall {\vv} \in \V\backslash\A, {\ww} \in \V,\label{c:8}\\
& \; z_{\vv} \in \{0,1\}, & \forall {\vv} \in \V,\label{c:9a}\\
& \; \x_{\vv} \in \R^{d}, & \forall {\vv} \in \V,\label{c:9b}
\end{align}

where $\V$ is any index set satisfying: $\A,\B\subset \V$, and $|\V|=\nu_{\A}(\B)$. The parameters $\Delta:=\dmax\{\|\a\|_{\infty}: \a\in \A\}$, and $\varepsilon$ is a small enough tolerance value. 
\end{theorem}
\begin{proof}
    In the above formulation, the following variables are used:
$$
z_{\vv} = \begin{cases}
    1 & \mbox{if ${\vv}$ is a vertex of the minimal mediated graph,}\\
    0 & \mbox{otherwise}
\end{cases},
$$
$$
y_{{\vv}{\ww}} = \begin{cases}
    1 & \mbox{if $({\vv},{\ww})$ is an arc of the minimal mediated graph,}\\
    0 & \mbox{otherwise}
\end{cases},
$$
$$
\x_{\vv} \in \R^d: \mbox{Coordinates of the vertices in the MinMG in case $z_{\vv}=1$}, \forall {\vv}\in \V.
$$

Constrains \eqref{c:2} and \eqref{c:3} determines that an arc in the mediated graph is possible just in case the two extremes are vertices of the graph. Constraints \eqref{c:4} and \eqref{c:5} assures the correct construction of a mediated graph, that is, if $({\vv},{\ww}), ({\vv},{\uu})$ are arcs in the mediated graph, then, ${\vv}$ is the midpoint of ${\ww}$ and ${\uu}$. Constraints \eqref{c:6} assures that each vertex, seen as a embedded coordinate in $\R^d$, is activated only once. Constraint \eqref{c:1} is derived from Remark \ref{rem:1} where  mediated graph is simplified to one with exactly two outgoing arcs from each \emph{proper} mediated vertex. Finally, Constraints \eqref{c:7} set as vertices of the mediated graph those in $\A \cup \B$. The minimality is assured by the minimization criterion on the number of mediated vertices \eqref{c:obj}.
\end{proof}

Note that the Manhattan norm constraints are nonlinear but they can  be standardly rewritten as linear constraints. Thus, the model above results in a MILP problem.\\

\noindent{\bf Optimization over Generalized Power Cones}

Minimal mediated graphs in $\R^d$ have a direct impact in the solution of conic optimization problems. More specifically, a generalized power cone program has the following shape:
\begin{align} \label{eq:gpcp}\tag{GPCP}
    \text{minimize } &\; {\c}^T{\x} + {\d}^T \z\nonumber\\
    \text{subject to }&\; A{\x}+B{\z}= {\bb}, \nonumber\\
                        &\; \|{\x}\|_p\leq {\z}^{\ba}\nonumber\\
                        & \; {\x}\in\R^n,\nonumber \\
                        &\; {\z}\in \R_+^d,\nonumber
                    \end{align}

where $p\in \R_{\geq 1}$, $\ba \in \Lambda_d := \{(\alpha_1, \ldots, \alpha_d) \in \R^+: \sum_{j=1}^d \alpha_j =1\}$ ($d$-dimensional standard simplex), $\|\cdot\|_p$ denotes the $\ell_p$-norm in $\R^{n}$, and for any $\z\in \R^d_+$ and $\ba \in \Lambda_d$, $\z^{\ba} = \prod_{j=1}^{d} z_j^{\alpha_j} = z_1^{\alpha_1} \cdots z_d^{\alpha_d}$ is the weighted geometric mean of $\z$ with weights $\ba$. 

The above problem, although belong to the family of convex conic problems, requires to be equivalently reformulated as a second-order cone problem to be numerically solved in practice (e.g., using the available off-the-shelf solvers). With solutions of \eqref{eq:cminmg} one can build a minimal extended reformulation of a GPCP into a SOCP. Let $\A_{m}:=\{\e_i: i=1,\ldots, m\}\cup\{\0\}\subset \R^{m}$ (where $\e_i$ stands for the $i$th unit vector), $\beta_p:=\frac{1}{p}$ and $\be_{\ba}:=(\alpha_1,\ldots, \alpha_{d-1})$. Then, each couple of mediated graphs $G_p=(V_p,A_p)\in \M_{\A_1}(\beta_p)$ and $G_{\ba}=(V_{\ba},A_{\ba})\in \M_{\A_{d-1}}(\be_{\ba})$ are in one-to-one correspondence to the following SOCP extended reformulation of \eqref{eq:gpcp}~\citep[see][for further details]{blanco2024minimal}:

\begin{align}
    \text{Minimize } & \; \c^T\x + \d^T \z & \nonumber\\
    \text{subject to } &\; A\x+B\z= \bb, \nonumber\\
                        &\; \dsum_{j=1}^{d} t_ j \leq \eta_{\be_{\ba}},\nonumber\\
                        &\; \xi_{j\beta_p}=x_j, \; \xi_{j1}=t_j, \; \xi_{j0}=\eta_{\be_{\ba}},  & j=1,\ldots,n,\nonumber \\
                        &\; \eta_{\e_i}=z_i, \; \eta_{\0}=z_d, & i=1,\ldots, d-1 \nonumber \\
                        &\; \xi_{jv}^2\leq \prod_{u\in \delta^+_p(v)}\xi_{ju} & j=1,\ldots, n; v\in V_p\backslash \A_p \nonumber\\
                        &\; \eta_{{\vv}}^2\leq \prod_{{\uu}\in \delta^+_{\ba}({\vv})}\eta_{{\uu}}, & {\vv}\in V_{\ba}\backslash \A_{\ba}, \nonumber\\
                        &\; \t, \z, \XI, \Eta \geq 0.\nonumber
\end{align}
The number of SOC constraints in the reformulation is $n(|V_p|-2)+|V_{\ba}|-d$, the number of variables is $n|V_p|+|V_{\ba}|$, and the number of linear constraints is $O(m+n|V_p|+|V_{\ba}|)$ taking into account the ones hiding in the SOC constraints. According with \cite{blanco2025complexity}, assuming that the coefficients of the input data $(A,B,\bb,\c,\d)$ have bit size at most $\tau$, then the feasibility of \eqref{eq:gpcp} can be tested in $m(n|V_p|+|V_{\ba}|)^{O(n|V_p|+|V_{\ba}|)}$ arithmetic operations over $\tau(n|V_p|+|V_{\ba}|)^{O(n|V_p|+|V_{\ba}|)}$-bit numbers. An $\varepsilon$-optimal solution of \eqref{eq:gpcp} can be obtained through binary search in $ (\tau + N)m(n|V_p|+|V_{\ba}|)^{O(n|V_p|+|V_{\ba}|)}$ arithmetic operations, where $N=2^{-\varepsilon}\in \Z_+$. So finding minimal mediated graphs is highly recommended to reduce the computational complexity of \eqref{eq:gpcp}. 

On the other hand, note that the solution of the problem allows to construct the minimal mediated graph, and then, the explicit SOCP representation of the problem. Specifically, let $(\x^*_p,\z^*_p,\y^*_p)$ be the solution of \eqref{eq:cminmg} for $\A_1$ and $\{\beta_p\}$, defining $V^*_p:=\{x^*_{pv} \in \R : z^*_{pv}=1\}$ and $A^*_p:=\{(x^*_{pv}, x^*_{pu})\in \R^2 : y^*_{pvu}=1\}$, then by Theorem \ref{th:continuous}, the mediated graph $G^*_p:=(V^*_p,A^*_p)\in \MinMG_{\A_1}(\beta_p)$. Likewise, let $(\x^*_{\ba},\z^*_{\ba},\y^*_{\ba})$ be the solution of \eqref{eq:cminmg} for $\A_{d-1}$ and $\{\be_{\ba}\}$, defining $V^*_{\ba}:=\{\x^*_{\ba {\vv}} \in \R^{d-1} : z^*_{\ba{\vv}}=1\}$ and $A^*_{\ba}:=\{(\x^*_{\ba{\vv}}, \x^*_{\ba{\uu}})\in \R^{d-1}\times \R^{d-1} : y^*_{\ba{\vv}{\uu}}=1\}$, then by Theorem \ref{th:continuous}, the mediated graph $G^*_{\ba}:=(V^*_{\ba},A^*_{\ba})\in \MinMG_{\A_{d-1}}(\be_{\ba})$. Hence, the SOCP defined by $G^*_p$ and $G^*_{\ba}$ is a minimal SOCP (also SDP) extended reformulation of \eqref{eq:gpcp}.

Similar applications of Theorem \ref{th:continuous} can be used to derive optimal SOCP reformulations in SONC optimization, matrix optimization, or quantum information. The reader is referred to~\citep{magron2023sonc,wang2024weighted} for further details.

\subsubsection*{Minimal Mediated Graphs in discrete domains}

The computation of discrete domain minimal mediated set, as far as we know, has not been previously addressed, and although the above model could be adapted (by enforcing that the $\x$-variables can only take feasible values in the finite set of points inside $\conv(\A) \cap \Z^d$ or $\conv(\A) \cap (2\Z)^d$, it would not exploit the discrete nature of the problem, and will result in an inefficient approach to compute $\MinMG$. In what follows we describe an alternative integer linear programming model (ILP) that we propose for the problem.
\begin{theorem}\label{th:discrete}
Let $\A, \B \subset \Z^d$ be two finite sets points with $\B \subset \conv(\A)$. The following integer linear programming model allows to compute a mediated graph in $\MinMG_\A(\B)$.
\begin{align}
    &&  \label{eq:dminmg}\tag{D-MinMGP}\\
    \text{Minimize } &\; |\A| + |\B|  + \sum_{{\vv}\in \V\backslash(\A\cup \B)} x_{\vv} & \label{minmg:obj}\\
    \text{subject to } & \; y_{{\vv}{\ww}} \leq x_{\vv}, &\forall {\vv}, {\ww}\in \V,\label{minmg:c1}\\
    & \; y_{{\vv}{\ww}} \leq x_{\ww}, &\forall {\vv}, {\ww} \in \V,\label{minmg:c2}\\
        & \; y_{{\vv} (2{\vv}-{\ww})} \geq y_{{\vv}{\ww}}, &\forall {\vv}\neq {\ww} \in \V \text{ if } 2{\vv}-{\ww} \in \V,\label{minmg:c3}\\
            &\; y_{{\vv}{\ww}} =0, &\forall {\vv}\neq {\ww} \in \V \text{ if } 2{\vv}-{\ww} \not\in \V,\label{minmg:c4}\\
    &\; \sum_{{\ww} \in V} y_{{\vv}{\ww}} = 2x_{\vv}, &\forall {\vv} \in \V\backslash\A,\label{minmg:c5}\\
    &\; x_{\vv} =1, &\forall {\vv} \in \A\cup \B,\label{minmg:c6}\\
    &\; x_{\vv} \in \{0,1\}, &\forall {\vv} \in \V, \label{minmg:c7}\\
    &\; y_{{\vv}{\ww}} \in \{0,1\}, &\forall {\vv}, {\ww} \in \V.\label{minmg:c8}
\end{align}
where $\V$ is the (finite) set of potential positions for the mediated vertices.
\end{theorem}
\begin{proof}
Defining the binary variables that completely identify a graph in $\M_\A(\B)$:
        $$
        x_{\vv} = \begin{cases}
            1 & \mbox{if ${\vv}$ is a mediated vertex in $\MinMG_\A(\B)$,}\\
            0 & \mbox{otherwise}
        \end{cases}
$$
        $$
y_{{\vv}{\ww}} = \begin{cases}
1 & \mbox{if $({\vv},{\ww})$ is an arc in $\MinMG_\A(\B)$,}\\
            0 & \mbox{otherwise}
        \end{cases}
        $$
These variables are adequately defined by contraints \eqref{minmg:c1}-\eqref{minmg:c8}. Constraints \eqref{minmg:c1} and \eqref{minmg:c2} assurre that no arcs are possible unless the extreme vertices are part of the graph. Constraints \eqref{minmg:c3} and \eqref{minmg:c4} imply the verification of the \emph{mediated} condition: if $({\vv},{\ww})$ is an arc in the mediated graph, then, $({\vv},2{\vv}-{\ww})$ is also an arc in case all the extreme are feasible vertices. Contraint \eqref{minmg:c4} states that all the mediated vertices (except those in $\A$) must have exactly two ongoing arcs. Constraints \eqref{minmg:c6}. The domain of the variables are given in Constraints \eqref{minmg:c7} and \eqref{minmg:c8}. The minimality is assured by the minimization criterion on the number of mediated vertices \eqref{minmg:obj}. Note that the elements in $\A$ are excluded from the sum of the variables since they are allways in the mediated graph, and then, its size is incorporated to the objective function as a constant (that can be avoided when solving the model).
\end{proof}
\begin{remark}
    The ILP model in the previous result has $O(|\V|^2)$ variables and $O(|\V|^2)$ linear constraints. For large sizes of $\V$ the problem can be computationally costly. Some strategies can be applied to the model in order to facilitate the solution procedure:
    \begin{itemize}
        \item Constraints \eqref{minmg:c1} are not necessary. They are already induced by Constraints \eqref{minmg:c5}. Note that in case $x_{\vv}=0$, all the outgoing arcs from ${\vv}$ are not allowed in the solution.
        \item Contraints \eqref{minmg:c2} can be strenthened by aggregating all the constraint for a given ${\ww}$, i.e., one can replace Constraints \eqref{minmg:c2} by
        \begin{align*}
            \sum_{{\vv} \in \V} y_{{\vv}{\ww}} \leq |\V| x_{\ww}, \forall {\ww} \in \V.
        \end{align*}
        \item The $y$-variables are not required to be defined for the case $y_{{\vv}{\vv}}$ for ${\vv}\in \V$. Note that loops are not allowed in a mediated graph. The same applies for the $x_{\vv}$-variables for $v\in \A \cup \B$, that are already fixed to one in our model
    \end{itemize}
\end{remark}
\begin{remark}\label{remark:allmin}
    Note that the ILP presented above allows to compute a single minimal mediated graph in the discrete domain case. It might be possible that more than one mediated graph achieves this minimum cardinality. One could obtain the whole list of mediated graphs in $\MinMG_\A(\B)$ by iteratively solving the problem and adding constraints to \textit{filter} the previously obtained mediated graphs until the optimal solution achieve an objective value strictly larger than the minimum cardinal. Specifically, if solving for the first time the problem we obtain a minimum cardinality of $\nu$, with sets of arcs (provided by the solutions in the $y$-variables) $A_{(1)}$, one can add the constraints:
\begin{align*}
    |\A| + |\B|  + \sum_{{\vv}\in \V\backslash(\A\cup \B)} x_{\vv} \leq \nu,\\
    \dsum_{{\vv}, {\ww} \in \V} y_{{\vv}{\ww}} \leq |A_{(1)}|-1
\end{align*}
and solve the problem again. Thus, in the $k$th iteration, when the set of arcs $A_{(k)}$ is obtained, we would add the constraint:
\begin{align*}
    \dsum_{{\vv}, {\ww} \in \V} y_{{\vv}{\ww}} \leq |A_{(k)}|-1
\end{align*}
In case the problem is infeasible, then the list of elements in $\MinMG_\A(\B)$ is already obtained, otherwise, a new constraint is incorporated and the problem is solved again.
\end{remark}

\begin{example}\label{ex:listMinMG}
    Let us consider the set $\A=\{(0,0), (7,0), (0,7)\}$ and $\B=\{(1,1)\}$. In Figure \ref{fig:lisMinMG} we show the five elements in $\MinMG_\A(\B)$ in case $\MM=\Z^2$. They where obtained in the same order as they as plotted by cutting off the solutions previously obtained. All of them have $10$ vertices. Observe that although the first and second, and the third and four mediated graphs are symmetric with respect to the line $\{x_1=x_2\}$, the three obtained \emph{proper} mediated graphs have a very different combinatorial structure.
    \begin{figure}[h]
\begin{tikzpicture}[scale=0.34]
\draw [step=1.0, gray, very thin] (0,0) grid (7,7);
\node (A1) at (0,7) [circle, draw, red, inner sep=0.8, fill=red] {};
\node (A2) at (2,4) [circle, draw, blue, inner sep=0.8, fill=blue] {};
\node (A3) at (4,0) [circle, draw, blue, inner sep=0.8, fill=blue] {};
\node (A4) at (1,2) [circle, draw, blue, inner sep=0.8, fill=blue] {};
\node (A5) at (0,0) [circle, draw, red, inner sep=0.8, fill=red] {};
\node (A6) at (1,1) [circle, draw, blue, inner sep=0.8, fill=blue] {};
\node (A7) at (7,0) [circle, draw, red, inner sep=0.8, fill=red] {};
\node (A8) at (2,0) [circle, draw, blue, inner sep=0.8, fill=blue] {};
\node (A9) at (1,0) [circle, draw, blue, inner sep=0.8, fill=blue] {};
\node (A10) at (4,1) [circle, draw, blue, inner sep=0.8, fill=blue] {};
\draw[-{Stealth[scale=0.4]}, line width=0.2] (A4)--(A5);
\draw[-{Stealth[scale=0.4]}, line width=0.2] (A10)--(A7);
\draw[-{Stealth[scale=0.4]}, line width=0.2] (A9)--(A8);
\draw[-{Stealth[scale=0.4]}, line width=0.2] (A8)--(A5);
\draw[-{Stealth[scale=0.4]}, line width=0.2] (A6)--(A9);
\draw[-{Stealth[scale=0.4]}, line width=0.2] (A6)--(A4);
\draw[-{Stealth[scale=0.4]}, line width=0.2] (A2)--(A1);
\draw[-{Stealth[scale=0.4]}, line width=0.2] (A3)--(A9);
\draw[-{Stealth[scale=0.4]}, line width=0.2] (A8)--(A3);
\draw[-{Stealth[scale=0.4]}, line width=0.2] (A9)--(A5);
\draw[-{Stealth[scale=0.4]}, line width=0.2] (A3)--(A7);
\draw[-{Stealth[scale=0.4]}, line width=0.2] (A2)--(A10);
\draw[-{Stealth[scale=0.4]}, line width=0.2] (A4)--(A2);
\draw[-{Stealth[scale=0.4]}, line width=0.2] (A10)--(A4);
\end{tikzpicture}~\begin{tikzpicture}[scale=0.34]
\draw [step=1.0, gray, very thin] (0,0) grid (7,7);
\node (A1) at (0,1) [circle, draw, blue, inner sep=0.8, fill=blue] {};
\node (A2) at (0,7) [circle, draw, red, inner sep=0.8, fill=red] {};
\node (A3) at (0,4) [circle, draw, blue, inner sep=0.8, fill=blue] {};
\node (A4) at (2,1) [circle, draw, blue, inner sep=0.8, fill=blue] {};
\node (A5) at (0,0) [circle, draw, red, inner sep=0.8, fill=red] {};
\node (A6) at (1,1) [circle, draw, blue, inner sep=0.8, fill=blue] {};
\node (A7) at (7,0) [circle, draw, red, inner sep=0.8, fill=red] {};
\node (A8) at (4,2) [circle, draw, blue, inner sep=0.8, fill=blue] {};
\node (A9) at (1,4) [circle, draw, blue, inner sep=0.8, fill=blue] {};
\node (A10) at (0,2) [circle, draw, blue, inner sep=0.8, fill=blue] {};
\draw[-{Stealth[scale=0.4]}, line width=0.2] (A6)--(A4);
\draw[-{Stealth[scale=0.4]}, line width=0.2] (A10)--(A5);
\draw[-{Stealth[scale=0.4]}, line width=0.2] (A6)--(A1);
\draw[-{Stealth[scale=0.4]}, line width=0.2] (A3)--(A1);
\draw[-{Stealth[scale=0.4]}, line width=0.2] (A1)--(A5);
\draw[-{Stealth[scale=0.4]}, line width=0.2] (A9)--(A2);
\draw[-{Stealth[scale=0.4]}, line width=0.2] (A4)--(A8);
\draw[-{Stealth[scale=0.4]}, line width=0.2] (A4)--(A5);
\draw[-{Stealth[scale=0.4]}, line width=0.2] (A10)--(A3);
\draw[-{Stealth[scale=0.4]}, line width=0.2] (A1)--(A10);
\draw[-{Stealth[scale=0.4]}, line width=0.2] (A8)--(A9);
\draw[-{Stealth[scale=0.4]}, line width=0.2] (A8)--(A7);
\draw[-{Stealth[scale=0.4]}, line width=0.2] (A3)--(A2);
\draw[-{Stealth[scale=0.4]}, line width=0.2] (A9)--(A4);
\end{tikzpicture}~\begin{tikzpicture}[scale=0.34]
\draw [step=1.0, gray, very thin] (0,0) grid (7,7);
\node (A1) at (0,7) [circle, draw, red, inner sep=0.8, fill=red] {};
\node (A2) at (2,1) [circle, draw, blue, inner sep=0.8, fill=blue] {};
\node (A3) at (0,0) [circle, draw, red, inner sep=0.8, fill=red] {};
\node (A4) at (1,1) [circle, draw, blue, inner sep=0.8, fill=blue] {};
\node (A5) at (7,0) [circle, draw, red, inner sep=0.8, fill=red] {};
\node (A6) at (4,2) [circle, draw, blue, inner sep=0.8, fill=blue] {};
\node (A7) at (1,4) [circle, draw, blue, inner sep=0.8, fill=blue] {};
\node (A8) at (2,3) [circle, draw, blue, inner sep=0.8, fill=blue] {};
\node (A9) at (2,2) [circle, draw, blue, inner sep=0.8, fill=blue] {};
\node (A10) at (3,2) [circle, draw, blue, inner sep=0.8, fill=blue] {};
\draw[-{Stealth[scale=0.4]}, line width=0.2] (A10)--(A6);
\draw[-{Stealth[scale=0.4]}, line width=0.2] (A8)--(A10);
\draw[-{Stealth[scale=0.4]}, line width=0.2] (A8)--(A7);
\draw[-{Stealth[scale=0.4]}, line width=0.2] (A7)--(A1);
\draw[-{Stealth[scale=0.4]}, line width=0.2] (A2)--(A6);
\draw[-{Stealth[scale=0.4]}, line width=0.2] (A2)--(A3);
\draw[-{Stealth[scale=0.4]}, line width=0.2] (A4)--(A9);
\draw[-{Stealth[scale=0.4]}, line width=0.2] (A7)--(A2);
\draw[-{Stealth[scale=0.4]}, line width=0.2] (A6)--(A7);
\draw[-{Stealth[scale=0.4]}, line width=0.2] (A6)--(A5);
\draw[-{Stealth[scale=0.4]}, line width=0.2] (A9)--(A8);
\draw[-{Stealth[scale=0.4]}, line width=0.2] (A4)--(A3);
\draw[-{Stealth[scale=0.4]}, line width=0.2] (A9)--(A2);
\draw[-{Stealth[scale=0.4]}, line width=0.2] (A10)--(A9);
\end{tikzpicture}~\begin{tikzpicture}[scale=0.34]
\draw [step=1.0, gray, very thin] (0,0) grid (7,7);
\node (A1) at (0,7) [circle, draw, red, inner sep=0.8, fill=red] {};
\node (A2) at (2,4) [circle, draw, blue, inner sep=0.8, fill=blue] {};
\node (A3) at (1,2) [circle, draw, blue, inner sep=0.8, fill=blue] {};
\node (A4) at (0,0) [circle, draw, red, inner sep=0.8, fill=red] {};
\node (A5) at (1,1) [circle, draw, blue, inner sep=0.8, fill=blue] {};
\node (A6) at (7,0) [circle, draw, red, inner sep=0.8, fill=red] {};
\node (A7) at (2,3) [circle, draw, blue, inner sep=0.8, fill=blue] {};
\node (A8) at (2,2) [circle, draw, blue, inner sep=0.8, fill=blue] {};
\node (A9) at (3,2) [circle, draw, blue, inner sep=0.8, fill=blue] {};
\node (A10) at (4,1) [circle, draw, blue, inner sep=0.8, fill=blue] {};
\draw[-{Stealth[scale=0.4]}, line width=0.2] (A3)--(A4);
\draw[-{Stealth[scale=0.4]}, line width=0.2] (A10)--(A6);
\draw[-{Stealth[scale=0.4]}, line width=0.2] (A9)--(A7);
\draw[-{Stealth[scale=0.4]}, line width=0.2] (A9)--(A10);
\draw[-{Stealth[scale=0.4]}, line width=0.2] (A5)--(A8);
\draw[-{Stealth[scale=0.4]}, line width=0.2] (A2)--(A1);
\draw[-{Stealth[scale=0.4]}, line width=0.2] (A7)--(A2);
\draw[-{Stealth[scale=0.4]}, line width=0.2] (A8)--(A3);
\draw[-{Stealth[scale=0.4]}, line width=0.2] (A7)--(A8);
\draw[-{Stealth[scale=0.4]}, line width=0.2] (A5)--(A4);
\draw[-{Stealth[scale=0.4]}, line width=0.2] (A3)--(A2);
\draw[-{Stealth[scale=0.4]}, line width=0.2] (A2)--(A10);
\draw[-{Stealth[scale=0.4]}, line width=0.2] (A8)--(A9);
\draw[-{Stealth[scale=0.4]}, line width=0.2] (A10)--(A3);
\end{tikzpicture}~\begin{tikzpicture}[scale=0.34]
\draw [step=1.0, gray, very thin] (0,0) grid (7,7);
\node (A1) at (0,1) [circle, draw, blue, inner sep=0.8, fill=blue] {};
\node (A2) at (0,7) [circle, draw, red, inner sep=0.8, fill=red] {};
\node (A3) at (4,0) [circle, draw, blue, inner sep=0.8, fill=blue] {};
\node (A4) at (0,4) [circle, draw, blue, inner sep=0.8, fill=blue] {};
\node (A5) at (0,0) [circle, draw, red, inner sep=0.8, fill=red] {};
\node (A6) at (1,1) [circle, draw, blue, inner sep=0.8, fill=blue] {};
\node (A7) at (7,0) [circle, draw, red, inner sep=0.8, fill=red] {};
\node (A8) at (2,0) [circle, draw, blue, inner sep=0.8, fill=blue] {};
\node (A9) at (0,2) [circle, draw, blue, inner sep=0.8, fill=blue] {};
\node (A10) at (1,0) [circle, draw, blue, inner sep=0.8, fill=blue] {};
\draw[-{Stealth[scale=0.4]}, line width=0.2] (A8)--(A5);
\draw[-{Stealth[scale=0.4]}, line width=0.2] (A9)--(A5);
\draw[-{Stealth[scale=0.4]}, line width=0.2] (A10)--(A8);
\draw[-{Stealth[scale=0.4]}, line width=0.2] (A4)--(A1);
\draw[-{Stealth[scale=0.4]}, line width=0.2] (A1)--(A5);
\draw[-{Stealth[scale=0.4]}, line width=0.2] (A6)--(A8);
\draw[-{Stealth[scale=0.4]}, line width=0.2] (A9)--(A4);
\draw[-{Stealth[scale=0.4]}, line width=0.2] (A3)--(A10);
\draw[-{Stealth[scale=0.4]}, line width=0.2] (A1)--(A9);
\draw[-{Stealth[scale=0.4]}, line width=0.2] (A4)--(A2);
\draw[-{Stealth[scale=0.4]}, line width=0.2] (A10)--(A5);
\draw[-{Stealth[scale=0.4]}, line width=0.2] (A3)--(A7);
\draw[-{Stealth[scale=0.4]}, line width=0.2] (A8)--(A3);
\draw[-{Stealth[scale=0.4]}, line width=0.2] (A6)--(A9);
\end{tikzpicture}
\caption{The entire set $\MinMG_\A(\B)$ of Example \ref{ex:listMinMG}.\label{fig:lisMinMG}}
\end{figure}
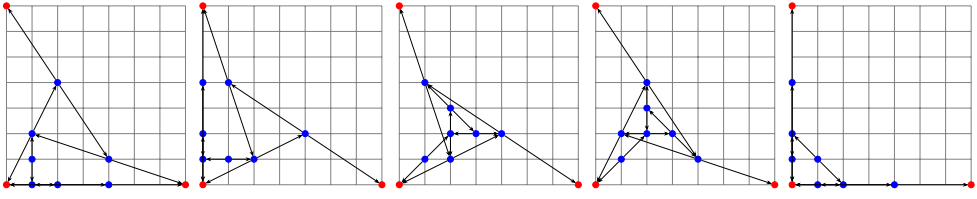
\end{example}
\medskip

\noindent{\bf SOS decomposition of SONC polynomials}. 

Let $\R[\x]=\R[x_1,\ldots,x_d]$ denote the ring of real $d$-variate polynomials. For $\ba=(\alpha_1,\ldots, \alpha_d)\in \Z_+^d$ , $\x^{\ba}:=x_1^{\alpha_1}\cdots x_d^{\alpha_d}$. Assume that $\A\subset (2\Z)^{d}$ is an affine independent set. A polynomial $f=\dsum_{\ba\in \A}c_{\ba}\x^{\ba}+c\x^{\be}\in \R[\x]$ is called a \emph{circuit polynomial} (\emph{circuit}, for short) if $c_{\ba}>0$, $c\neq 0$, and $\be\in\conv(\A)^\circ$~\citep{iliman2016amoebas}. Let $\lambda_{\ba}$ be the barycentric coordinates of $\be$ with respect to $\A$. The \emph{circuit number} $\Theta_f$ is defined as
\begin{equation}\label{eq:circuitnumber}
    \Theta_f=\dprod_{\ba\in\A}\left(\frac{c_{\ba}}{\lambda_{\ba}}\right)^{\lambda_{\ba}}.
\end{equation}
The nonnegativity of the circuit polynomial $f$ can be checked by its circuit number: $f$ is nonnegative if and only if either $|c|\leq \Theta_f$ and $\be\notin (2\Z)^d$, or $c\geq -\Theta_f$ and $\be\in (2\Z)^d$. 

Let $\B\subset \Z^d\cap \conv(\A)^\circ$ be a finite set, a polynomial $f\in \R[\x]$ is a sum of nonnegative circuits (SONC) polynomial if $f=\dsum_{\be\in \B}f_{\be}$ where $f_{\be}$ nonnegative circuit polynomial supported on $\A\cup\{\be\}$. Given a mediated graph $G=(V,A)\in \M_{\A}(\B)$, the construction of a sum of squares (SOS) decomposition -in fact, Sum of Binomial Squares (SOBS)- is based on mediated graphs in the intersection of $\M_{\A}(\be)$ and the set of submediated graphs of $G$ for each $\be\in \B$. In this shape each range of mediated graphs $G_{\be}=(V_{\be},A_{\be})\in \M_{\A}(\be)$ that is subgraph of $G$ for $\be\in \B$ defines a SOS decomposition of $f=\dsum_{\be\in \B}f_{\be}$ as $f=\dsum_{\be\in \B}\sigma_{\be}$ where $f_{\be}=\sigma_{\be}$ is a SOS decomposition of $f_{\be}$ whose number of squares is $|V_{\be}|-(d+1)$.

However, the authors of this paper realize this claim on the number of squares just follows directly from \citep[Theorem 5.2]{iliman2016amoebas} in the extremal case, when the coefficient of $\be$ in $f_{\be}$ is $\pm\Theta_{f_{\be}}$, because of the inner case is derived from a convexity argument from the extremal cases. Whereas the theoretical result of the characterization of the intersection of SONC and SOS by means of the combinatorial structure of the Newton polytope is perfect, the authors of this paper consider it necessary to demonstrate an explicit SOS decomposition of SONC polynomials that use the original support and coefficients of the polynomial.

First, we define the global minimizer of a nonnegative circuit $\s^*_f\in \R^d$ as the unique vector satisfying
\begin{equation}
\dprod_{k=1}^d\left(e^{s^*_{fk}}\right)^{\alpha(j))_k-\alpha(0)_k}=e^{\langle \s^*_f,\overrightarrow{\ba(0)\ba(j)}\rangle}=\frac{\lambda_jc_0}{\lambda_0c_j}
\end{equation}
for all $j=1,\ldots,d$, where $\A=\{\ba(0),\ldots,\ba(d)\}$.

\begin{theorem}\label{the:sobs}
    Let $\A=\{\ba(0),\ldots,\ba(d)\}$, and $f(\x)=\dsum_{j=0}^dc_j\x^{\ba(j)}+c\x^{\be}$ be a nonnegative circuit polynomial and $G=(V,A)\in \M_{\A}(\be)$ be a mediated graph. Then,
    \begin{enumerate}
        \item if $\be\notin (2\Z)^d$,
    \begin{align*}
        f(\x)&=c_{\be\be}\frac{c_0}{\lambda_0}e^{\langle \s^*_f,\ba(0)\rangle}\dsum_{m=0}^1 \frac{\Theta_{f}+(-1)^mc}{2\Theta_{f}}\left(\frac{\x^{\frac{\be'}{2}}}{\sqrt{e^{\langle \s^*_f,\be'\rangle}}}+ (-1)^m\frac{\x^{\frac{\be''}{2}}}{\sqrt{e^{\langle \s^*_f,\be''\rangle}}}\right)^2\\
        & + \frac{c_0}{\lambda_0}e^{\langle \s^*_f,\ba(0)\rangle}\dsum_{\ga\in V\setminus (\A\cup \{\be\})} c_{\be\ga}\left(\frac{\x^{\frac{\ga'}{2}}}{\sqrt{e^{\langle \s^*_f,\ga'\rangle}}}-\frac{\x^{\frac{\ga''}{2}}}{\sqrt{e^{\langle \s^*_f,\ga''\rangle}}}\right)^2;
    \end{align*}

    \item if $\be\in (2\Z)^d$ and $c <0$,
\begin{align*}
        f(\x)&=(\Theta_{f}+c)\left(\x^{\frac{\be}{2}}\right)^2\\
        &+ \frac{c_0}{\lambda_0}e^{\langle \s^*_f,\ba(0)\rangle}\dsum_{\ga\in V\setminus \A} c_{\be\ga}\left(\frac{\x^{\frac{\ga'}{2}}}{\sqrt{e^{\langle \s^*_f,\ga'\rangle}}}-\frac{\x^{\frac{\ga''}{2}}}{\sqrt{e^{\langle \s^*_f,\ga''\rangle}}}\right)^2;
    \end{align*}
    \end{enumerate}
        
where $\ga',\ga''\in \delta^+(\ga)$, $c_{\be\ga}\geq0$ is the entry $(\be,\ga)$ in $(L^+(G)_{V\backslash \A})^{-1}$ for every $\ga\in V\backslash \A$.
\end{theorem}

For the sake of the reader the proof of Theorem \ref{the:sobs} might be found in the Appendix \ref{appendix}.

In both cases, the number of squares is $|V|-d$. Then, let $\A\subset (2\Z)^d$ be an affine independent set, $\B\subset \Z^d\cap \conv(\A)^\circ$ be a finite set, and let $f\in\R[\x]$ be a SONC supported on $\A\cup \B$. Given a mediated graph $G=(V,A)\in \M_{\A}(\B)$, the construction  and a range of mediated graphs $G_{\be}=(V_{\be},A_{\be})\in \M_{\A}(\be)$ that are subgraphs of $G$ for $\be\in \B$ then $f=\dsum_{\be\in \B}f_{\be}$ where $f_{\be}$ nonnegative circuit polynomial supported on $\A\cup\{\be\}$ can be written as $f=\dsum_{\be\in \B}\sigma_{\be}$ where $f_{\be}=\sigma_{\be}$ is the SOS decomposition of Theorem \ref{the:sobs}. Notice that the involved binomials can be $\left(s_{\ga} \x^{\frac{\ga'}{2}}-t_{\ga}\x^{\frac{\ga''}{2}}\right)$ and $\left(s_{\be} \x^{\frac{\be'}{2}}+t_{\be}\x^{\frac{\be'}{2}}\right)$ for $\ga\in V\backslash \A$, $\ga',\ga''\in\delta^+(\ga)$, $s_{\ga}, t_{\ga}\in \R_+$ if $\be\notin (2\Z)^d$; and $\left(\x^{\frac{\be}{2}}\right)$ if $\be\in (2\Z)^d$ so the number of squares of the decomposition is less or equal to $|V|+|\B|-d-1$, being equality if the coefficient of $\be$ in $f_{\be}$ is different from $\pm \Phi_{f_{\be}}$ for every $\be\in \B$. The number of squares decreases in one unit for each $\be\in \B$ that satisfies the equality. Hence, let $(\x^*_f,\y^*_f)$ be the solution of \eqref{eq:dminmg} for $\A$ and $\B$, defining $V^*_f:=\{\ga \in \Z^{d} : x^*_{f\ga}=1\}$ and $A^*_f:=\{(\ga, \ga')\in \Z^{d}\times (2\Z)^{d} : y^*_{f\ga,\ga'}=1\}$, then by Theorem \ref{th:discrete}, the mediated graph $G^*_f:=(V^*_{f},A^*_{f})\in \MinMG_\A(\B)$. Thus, we can apply Theorem \ref{the:sobs} to each $f_{\be}$ with $G^*_f$ providing a SOS decomposition of $f$ which reduces the number of squares and the support of the decomposition, expanding the sparsity and reducing the rank of the associated positive semidefinite gram matrix of $f$.

\subsection{Maximal Mediated Graphs}\label{sec:maximal}

In what follows we analyze maximal mediated graphs, i.e., those mediated graphs for $\A$ that have maximum vertices cardinality. Note that in the \emph{continuous} case, the set of potential vertices for the mediated graph is not finite. Thus, in this section we focus on maximal mediated graphs for discrete domains. Note that in discrete domains, the number of vertices of any $\A$-mediated graph can be at most $|\conv(\A)\cap \MM|<\infty$ ensuring the existence of maximal elements.

\begin{definition}
Let $\A\subset \MM$ be a finite set. $G =(V,A) \in \M_\A$, is a Maximal $\A$-Mediated Graph, if there not exists $G' = (V',A')\in \M_\A$ such that $|V'|>|V|$.
\end{definition}
We denote by $\MaxMG_\A$ the family of maximal $\A$-mediated graphs. Notice that a maximal mediated set is a maximal element in the poset $(\M_\A,\preceq)$.

The first observation that we address is the unification of notation of maximal mediated sets. The maximal $\A$-mediated set was defined as the $\A$-mediated set that contains every $\A$-mediated set~\citep{hartzer2022initial}, i.e., maximality understood by the order induced by the set inclusion. Indeed, this kind of set are maximum for this order in the set of all $\A$-mediated sets.  In what follows we state the equivalence between our definition of maximal mediated graph and those of maximal mediated sets in the literature.

\begin{theorem}\label{theo:maximal}
    $G=(V,A)\in \M_\A$ is a maximal mediated graph if and only if $V$ is the maximal mediated set. Furthermore, if $G=(V,A)$ and $G'=(V',A')$ are maximal, then $V=V'$.
\end{theorem}

\begin{proof}
Let $G$ be a maximal mediated graph. If there is a $\A$-mediated set $V'$ such that $V'\nsubseteq V$, then $V\subset V\cup V'$. Since, $V\cup V'$ is a $\A$-mediated set there is a $\A$-mediated graph $G'=(V\cup V', A')$. However, $G'\succ G$ and that is not possible by the maximality of $G$. So, $V$ is the maximal mediated set.

On the other hand, let $V$ be the maximal $\A$-mediated set. Assume there is $G'=(V',A')\succeq G$, that implies either $|V'|>|V|$ or $G'=G$. However, the first condition cannot be true for the maximality of $V$ so $G'=G$. Hence, $G$ is maximal.
\end{proof}

In what follows we provide our mathematical optimization formulation that we propose to compute the maximal mediated graph for a set $\A$.
\begin{theorem}\label{th:max_discrete}
Let $\A \subset \Z^d$ be a finite set of point. The following integer linear programming model allows to compute the mediated graph in $\MaxMG_\A$.
\begin{align}\label{eq:dmaxmg}\tag{D-MaxMGP}
    && \nonumber \\
    \text{Maximize } &\; |\A|  + \sum_{{\vv}\in \V\backslash\A)} x_{\vv} & \label{maxmg:obj}\\
    \text{subject to } & \; \sum_{{\vv}\in \V} y_{{\vv}{\ww}} \leq |\V| x_{\vv}, &\forall {\ww}\in \V,\label{maxmg:c1}\\
    & \; \sum_{{\ww} \in \V}  y_{{\vv}{\ww}} \leq |\V| x_{\vv}, &\forall {\vv} \in \A,\label{maxmg:c2}\\
        & \; y_{{\vv} (2{\vv}-{\ww})} \geq y_{{\vv}{\ww}}, &\forall {\vv}\neq {\ww} \in \V \text{ if } 2{\vv}-{\ww} \in \V,\label{maxmg:c3}\\
            &\; y_{{\vv}{\ww}} =0, &\forall {\vv}\neq {\ww} \in \V \text{ if } 2{\vv}-{\ww} \not\in \V,\label{maxmg:c4}\\
    &\; \sum_{{\ww} \in V} y_{{\vv}{\ww}} = 2x_{\vv}, &\forall {\vv} \in \V\backslash\A,\label{maxmg:c5}\\
    &\; x_{\vv} =1, &\forall {\vv} \in \A\cup \B,\label{maxmg:c6}\\
    &\; x_{\vv} \in \{0,1\}, &\forall {\vv} \in \V, \label{maxmg:c7}\\
    &\; y_{{\vv}{\ww}} \in \{0,1\}, &\forall {\vv}, {\ww} \in \V.\label{maxmg:c8}
\end{align}
where $\V = \conv(\A) \cap \Z^d$.
\end{theorem}
\begin{proof}
    Note that the formulation is similar to the one provided in Theorem \ref{th:discrete}. Instead of maximizing the number of vertices in the obtained mediated graph, this cardinality is maximized. 
\end{proof}

\noindent{\bf Intersection of SONC and SOS cones}

Let $\A\subset (2\Z)^d$ be a finite affine independent set, and a finite set $\B\subset \Z^d\cap \conv(\A)^\circ$. \cite{hartzer2022initial} showed a complete characterization of the polynomials supported on the circuit $\A\cup \B$ that lie in the intersection of SONC and SOS cones. That result can be adapted in terms of the solution of \eqref{eq:dmaxmg}.

\begin{corollary}\label{th:sossonc}
    Let $f\in\R[\x]$ be a $\SONC$ supported on $\A\cup \B$ and $\x^*$ the solution of \eqref{eq:dmaxmg} for $\A$. Then, $f$ is $\SOS$ if and only if either $\x^*_{\be}=1$ or $\be\in (2\Z)^d$ and its coefficient is positive for every $\be\in\B$.
\end{corollary}

\begin{proof}
    The result is straightforward from \cite[Theorem 3.9]{hartzer2022initial} and Theorem \ref{th:max_discrete}.
\end{proof}

\section{Computational Experiments}\label{sec:exp}

In this section we report the results of our computational experience, that we performed to validate the proposals and compare with other approaches in the literature. 

\subsection{Data}

We use the datasets provided by \cite{hartzer2022initial} to analyze their proposed algorithm to compute maximal mediated sets and that the authors made publicly available at {\uu}rl{https://polymake.org/downloads/MMS/csv/}. The dataset contains simplicial sets, $\Delta = \{{\vv}_0,{\vv}_1, \ldots, {\vv}_d\}$, in dimensions $d \in \{2,3,4,5,6,7,8,9\}$ of different sizes, measured by their degrees:
$$
m(\Delta) := \min\left\{\rho \in \Z_+: \sum_{l=1}^d w_l \leq 2\rho \text{ for all } {\ww} \in \conv(\Delta)\cap(2\Z)^d\right\}
$$
The authors generate all the simplicial sets and lattices for the some of the dimensions and degrees, and sampled these sets for the rest of dimensions and higher degrees. In total, $8,941,852$ maximal mediated sets where computed in \citep{hartzer2022initial}, distributed by dimension as shown in Table \ref{table_datasets}. We randomly select some of these instances to compute the minimal and the maximal mediated graph using our approach. In the third column of Table \ref{table_datasets}, we detail the number of random instances selected for each combination of $d$ and $m$.

\begin{table}[h!]
\centering
\tiny
\begin{tabular}{p{0.2cm}|p{0.3cm}p{0.3cm}p{0.24cm}p{0.2cm}p{0.3cm}p{0.2cm}p{0.3cm}p{0.3cm}p{0.3cm}p{0.3cm}p{0.3cm}p{0.3cm}p{0.3cm}p{0.3cm}p{0.07cm}p{0.3cm}p{0.3cm}p{0.3cm}}
\hline
\textbf{d}      & 2  & 2   & 2   & 3  & 3  & 4  & 4  & 4  & 4  & 4  & 5  & 5  & 6  & 6  & 7  & 7  & 8  & 9  \\
\textbf{m}      & 50 & 100 & 150 & 10 & 16 & 6  & 8  & 10 & 14 & 16 & 8  & 16 & 16 & 20 & 4  & 16 & 16 & 16 \\
\textbf{\#} & 1000 & 1000 & 300  & 500 & 1000 & 150 & 1000 & 1000 & 1000 & 1000 & 1000 & 1000 & 1000 & 1000 & 15  & 1000 & 1000 & 300 \\
\hline
\end{tabular}
\caption{Size for the samples of the simplices from \citep{hartzer2022initial}.\label{table_datasets}}
\end{table}
We compute the maximal mediated graphs of the selected random sets $\A$, as detailed above. We also run the enumeration-and-filtering algorithm proposed in \citep{hartzer2022initial} to compare the efficiency of our approach with respect to the only previous proposal for this task.

For the minimal mediated graphs, a set of interior points $\B$ is required to obtain \emph{proper} subgraphs. We then run our model for a random selection of $s \in \{1,3,5\}$ integer points inside the convex hull of the points in $\A$ to analyze the performance on different number of points in $\B$. 

The models were coded in \texttt{Python} and solved using the optimization solver \texttt{Gurobi 12.0} on a Apple M1 Max with 64 GB RAM.

\subsection{Results on Maximal Mediated Graphs}

We run our MILP model for all the samples of the instances provided in Table \ref{table_datasets}. In total 15110 problems where solved. We also run the algorithm proposed in \citep[Algorithm 4.3]{hartzer2022initial}. 

In Figure \ref{fig:pp} we show the performance profiles for obtaining the maximal mediated graph both our mathematical optimization-based approach and the algorithm proposed in \citep{hartzer2022initial} (named \texttt{BM} and \texttt{HRWY}, for the initials of the authors' last names). In the plots we represent percent of instances solved in less than each unit of CPU time (in log-scales to ease the distinction). For a fair comparison, we do not include in those times the time required to enumerate the points inside the given lattices. It is evident that the performance of our approach is superior than the one in \citep{hartzer2022initial}. Although solving a MILP may require, in worst case, the enumeration of the whole set of integer feasible points, the solutions methods for these problems are designed to avoid such an enumeration by branching, bounding, and reducing the feasible region by cutting off part of the space. 

The summary of the obtained results is shown in Table \ref{t:max}, where we report, for each value of $d$ and $m$, the average CPU times (in seconds) with both our methodology (\texttt{time\_BM}) and the enumerative methodology (\texttt{time\_HRWY}) only for those instances that were optimally solved with each procedure. The percent of instances in each row that were not optimally solved within one hour with the enumerative methodology is reported in column \texttt{unsolved\_HRWY}. Our methodology was capable to solve optimally all the instances within the time limit.

\begin{table}
\centering

\begin{tabular}{| l | l | l | l | l |}
\hline
$d$ &$m$ & \texttt{time\_BM}& \texttt{time\_HRWY} & \texttt{unsolved\_HRWY} \\
\hline
\textbf{2} & \textbf{50} & 0.05 & 5.71 &0\% \\
\hline
 & \textbf{100} & 1.97 & 373.61 & 4.30\% \\
\hline
 & \textbf{150} & 23.91 & 818.49 & 46.67\% \\
\hline
\textbf{3} & \textbf{10} & 0.00 & 0.01 &0\% \\
\hline
 & \textbf{16} & 0.01 & 0.20 &0\% \\
\hline
\textbf{4} & \textbf{6} & 0.00 & 0.00 &0\% \\
\hline
 & \textbf{8} & 0.00 & 0.00 &0\% \\
\hline
 & \textbf{10} & 0.00 & 0.04 &0\% \\
\hline
 & \textbf{14} & 0.01 & 0.83 &0\% \\
\hline
 & \textbf{16} & 0.02 & 1.63 &0\% \\
\hline
\textbf{5} & \textbf{8} & 0.00 & 0.02 &0\% \\
\hline
 & \textbf{16} & 0.04 & 4.60 & 0.10\% \\
\hline
\textbf{6} & \textbf{16} & 0.09 & 32.05 & 0.30\% \\
\hline
 & \textbf{20} & 2.09 & 197.51 & 9.90\% \\
\hline
\textbf{7} & \textbf{4} & 0.00 & 0.01 &0\% \\
\hline
 & \textbf{16} & 3.16 & 58.28 & 2.20\% \\
\hline
\textbf{8} & \textbf{16} & 5.04 & 130.39 & 6.80\% \\
\hline
\textbf{9} & \textbf{16} & 22.51 & 307.61 & 17.67\% \\
\hline
\end{tabular}
\caption{Summary of our computational experiments for the maximal mediated graph.\label{t:max}}
\end{table}

\begin{figure}[h!]
\includegraphics[width=0.23\textwidth]{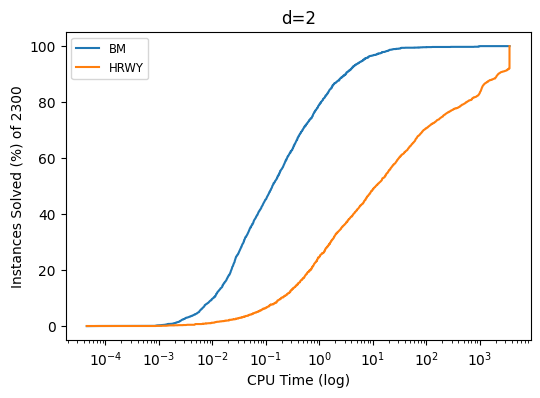}~\includegraphics[width=0.23\textwidth]{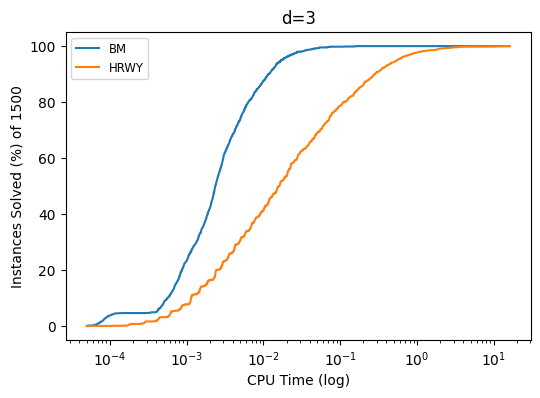}~\includegraphics[width=0.23\textwidth]{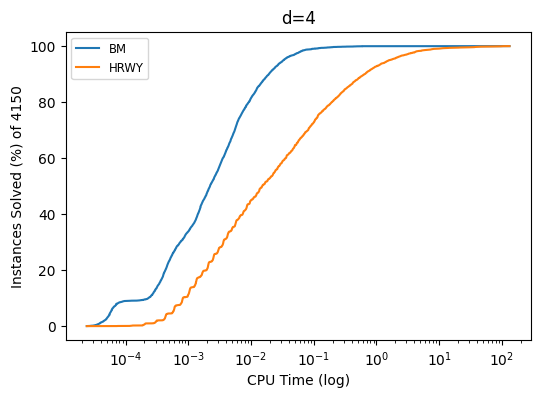}~\includegraphics[width=0.23\textwidth]{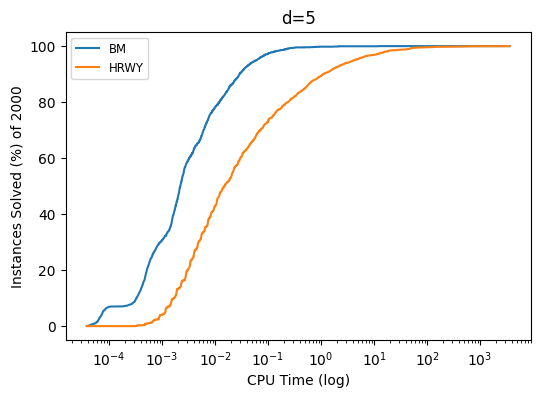}\\
\includegraphics[width=0.23\textwidth]{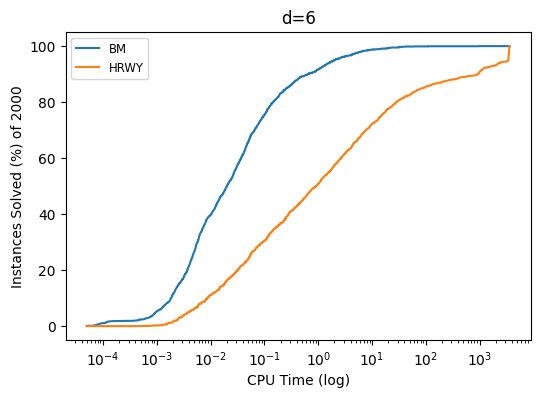}~\includegraphics[width=0.23\textwidth]{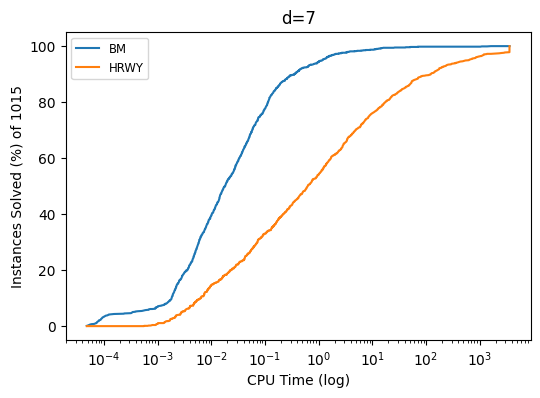}~\includegraphics[width=0.23\textwidth]{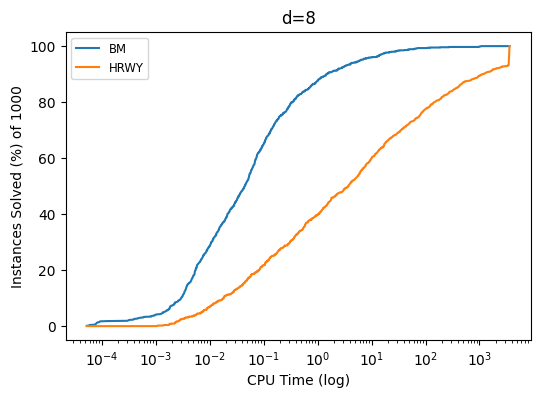}~\includegraphics[width=0.23\textwidth]{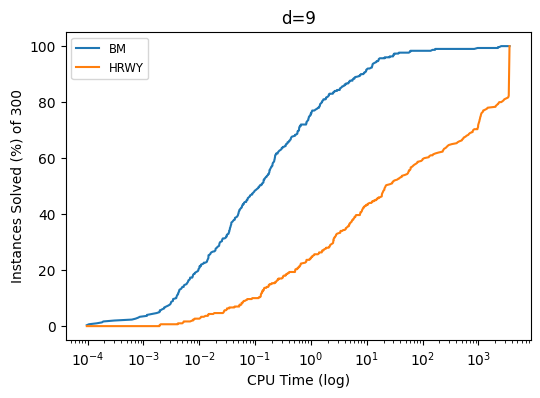}
\caption{Performance profiles for our approach and the one in \cite{hartzer2022initial}.\label{fig:pp}}
\end{figure}

In Figure \ref{fig:boxplot} we show, for each dimension $d$, the average CPU times (log-scale) to compute the maximal mediated graph with the optimization-based methodology that we propose (\texttt{BM}) and with the enumerative strategy proposed in \citep{hartzer2022initial} (\texttt{HRWY}). We also represent in the same plot, the average CPU times to enumerate the points inside the simplex, required to compute the optimal mediated graph with both methodologies.

In view of the results we conclude that the CPU time required by our optimization based methodology to obtain the maximal mediated graphs is significantly smaller than the enumerative procedure proposed in   \citep{hartzer2022initial}, and in some cases even able to construct the optimal solutions when the enumerative approach is not capable to do it.

Although the construction involved in the enumerative approach can be useful to understand the geometry of the mediated graphs, our approach has the advantage that it allows to solve then larger instances in reasonable CPU time. The combination of our approach and the reduction strategies in the enumerative approach by incorporating conditions in the form of linear inequalities in our model could result in a better approach, that will be explored in a forthcoming paper.

\begin{figure}[h!]
\includegraphics[width=0.24\textwidth]{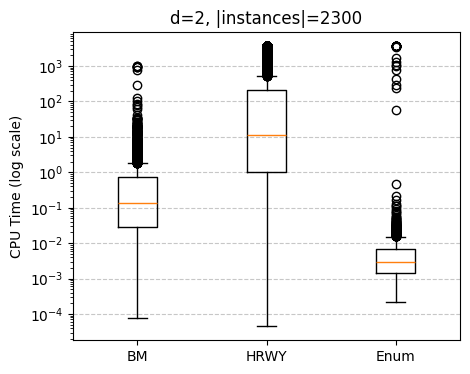}~\includegraphics[width=0.24\textwidth]{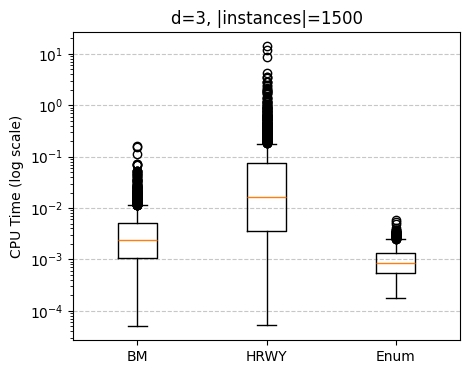}~\includegraphics[width=0.24\textwidth]{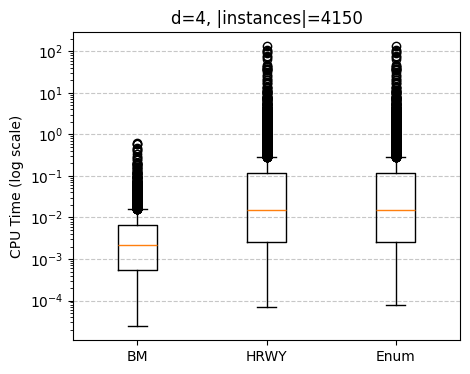}~\includegraphics[width=0.24\textwidth]{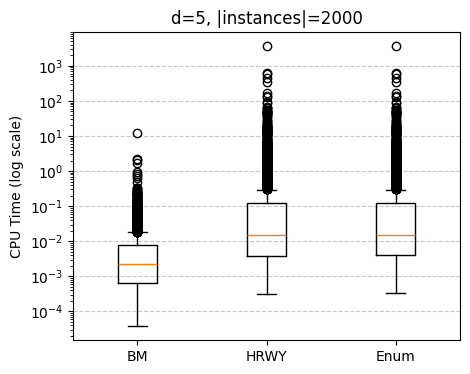}\\
\includegraphics[width=0.24\textwidth]{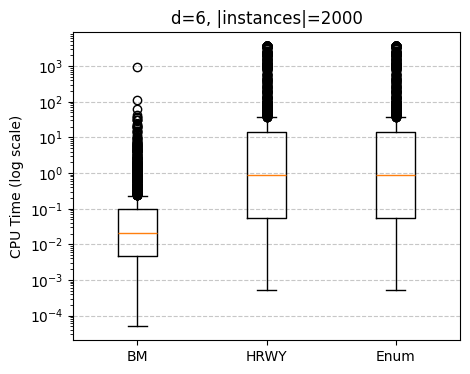}~\includegraphics[width=0.24\textwidth]{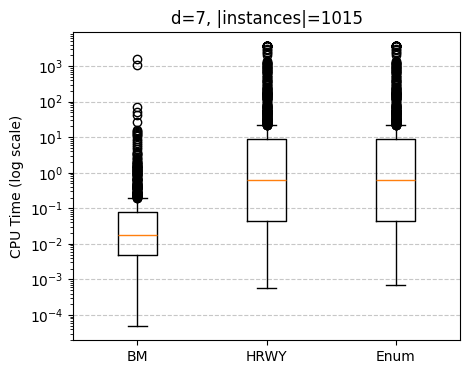}~\includegraphics[width=0.24\textwidth]{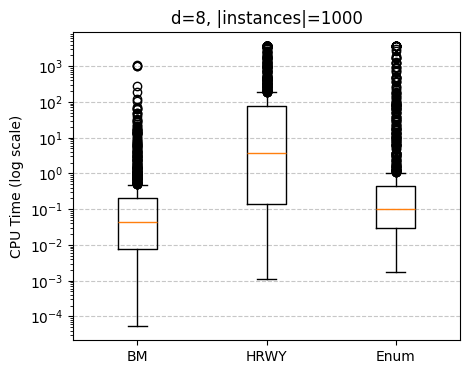}~\includegraphics[width=0.24\textwidth]{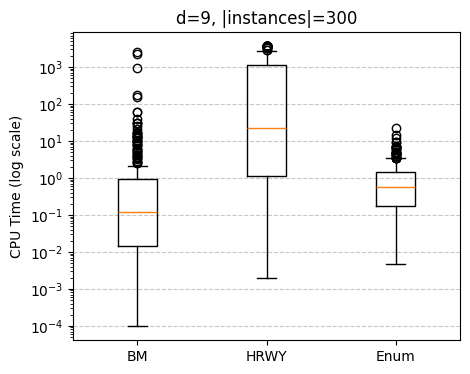}
\caption{Boxplots for the CPU times (in log-scale) required to compute the maximal mediated graphs with our methodology (\texttt{BM}), the methodology in \citep{hartzer2022initial}, and the CPU time required to enumerate the required points.\label{fig:boxplot}}
\end{figure}

\subsection{Results on Minimal Mediated Graphs}

For the minimal mediated graphs we randomly select $20$ random instances from each row, i.e., or each dimension $d$. For each of them, we randomly select $s$ interior points to the simplex $\A$, for $s \in \{1,3, 5\}$. These points form the set $\B$. Then, we consider three different domains for the vertices in the minimal mediated graph $\MM \in \{\R^d, \Z^d, (2\Z)^d\}$. In total, we solve $270$ instances with each of the three optimization models that we propose. We run the extended approach highlighted in Remark \ref{remark:allmin} to construct the whole set of minimal mediated graph for given sets $\A$ and $\B$ on an specif domain $\MM$.

For $\MM=\R^d$ (continuous domain), since the mathematical optimization problem is affected by the number of potential mediated vertices in the graph (upper bounded by $\nu_\A(\B)$ which can be large), following the suggestion in \citep{blanco2024minimal}, we solve the problems iteratively. Specifically, we start by taking $\mathcal{V}$ an index set with cardinal $1$, and solve the problem. In case the problem is feasible, we are done, and the solution is a minimal mediated graph. Otherwise, we increase the cardinal of $\mathcal{V}$ and repeat the process until a feasible solution is found. In most of the cases the cardinal of the minimal mediated graph is much smaller than the upper bound and we avoid solving huge integer linear problems. In general, we empirically tested that checking infeasibility requires much less time than solving a single but larger integer linear problem. 

In Table \ref{tab:minmg} we summarize the results of our experiments. There, the first two columns indicate the dimension ($d$) and the number of random points chosen in the set $\B$. Then, for each of the three domains $\MM \in \{\R^d, \Z^d, (2\Z)^d\}$, and then, different approaches we report the average sizes of the minimal mediated graph, the CPU time (in seconds) required to compute one minimal mediated graph, the average CPU time (in seconds) per extra minimal mediated graph  computed, and the number of minimal mediated graphs obtained. For the domain $(2\Z)^d$, either a single minimal mediated set was found or the problem was infeasible. Thus, we do not report the value \emph{Av. Time Rest} for this domain since all the values are zero.

As expected, as the domain is larger, the problem becomes more challenging. Although one could also expect that al larger the size of $|\B|$ more difficult the problem, this is not always true, as can be seen for $d=2$, where the problem for $s=3$ is more time demanding than the others. The same happens for the computation of more than one minimal mediated graph in $\MinMG_\A(\B)$. 

Regarding the size of the obtained minimal mediated graph, the number of vertices is slightly larger for $(2\Z)^d$ than for $\Z^d$, and the same happens for $\Z^d$ and $\R^d$. Note that although the approaches are different, all solve the same problems but in different, but nested, feasible regions. In fact, feasible solutions in the domain $(2\Z)^d$, are also feasible for $\Z^d$, and those for $\Z^d$ are feasible for $\R^d$. This observation also affects the number of graphs in $\MinMG_\A(\B)$, where for the domain $\R^d$ a larger numbers of minimal medited graphs were found.

Note that the case $d=2$ is particular, since in the dataset in \citep{hartzer2022initial}, the authors consider simplices with very large extremes compared to the other. Thus, the planar case can be more time demanding than the higher dimensional, because of the degree of the simplices that have being tested.

\begin{table}[h!]
\begin{tabular}{|p{0.02cm}p{0.2cm}|lll|lll|ll|lll|}\hline
\multirow{2}{*}{$d$} &\multirow{2}{*}{$|\mathcal{B}|$} & \multicolumn{3}{c|}{\# Vertices} & \multicolumn{3}{c|}{Time First Sol.} & \multicolumn{2}{c|}{Av. Time Rest} & \multicolumn{3}{c|}{$|\MinMG_\A(\B)|$}\\
  &  & $(2\Z)^d$ & $\Z^d$ & $\R^d$ & $(2\Z)^d$ & $\Z^d$ & $\R^d$ & $\Z^d$ & $\R^d$ & $(2\Z)^d$ & $\Z^d$ & $\R^d$ \\\hline
 \multirow{3}{*}{2}    & 1 & 5.24 & 5.24 & 4.53 & 0.02 & 0.15 & 43.32 & 0.32 & 0.01 & 1.00 & 1.06 & 1.29  \\
                       & 3 & 7.83 & 7.78 & 7.06 & 0.04 & 1.35 & 55.53 & 4.86 & 2.51 & 1.00 & 2.94 & 17.56 \\
                       & 5 & 10.39 & 10.11 & 8.94 & 0.07 & 0.91 & 9.78 & 1.91 & 0.30 & 1.00 & 6.44 & 17.28  \\\hline
\multirow{3}{*}{3}     & 1 & 5.18 & 5.18 & 5.18 & 0.01 & 0.03 & 0.01 & 0.02 & 0.00 & 1.00 & 1.00 & 1.00  \\
                       & 3 & 8.00 & 7.73 & 7.73 & 0.01 & 0.06 & 0.86 & 0.11 & 0.29 & 1.00 & 1.45 & 2.64  \\
                       & 5 & 10.55 & 10.45 & 10.45 & 0.01 & 0.11 & 16.10 & 0.07 & 3.62 & 1.00 & 3.36 & 10.73  \\\hline
  \multirow{3}{*}{4}   & 1 & 6.00 & 6.00 & 5.40 & 0.01 & 0.18 & 13.83 & 0.26 & 0.00 & 1.00 & 1.00 & 1.00  \\
                       & 3 & 8.60 & 8.60 & 7.80 & 0.02 & 0.26 & 15.32 & 0.33 & 0.02 & 1.00 & 1.30 & 1.40  \\
                       & 5 & 10.40 & 10.40 & 9.40 & 0.03 & 0.33 & 131.70 & 0.33 & 0.01 & 1.00 & 1.40 & 1.40  \\\hline
  \multirow{3}{*}{5}   & 1 & 7.13 & 7.13 & 7.13 & 0.00 & 0.12 & 0.01 & 0.18 & 0.00 & 1.00 & 1.00 & 1.00  \\
                       & 3 & 9.39 & 9.26 & 9.26 & 0.01 & 0.15 & 0.02 & 0.20 & 0.01 & 1.00 & 1.10 & 1.19 \\
                       & 5 & 11.10 & 11.39 & 11.39 & 0.01 & 0.18 & 0.05 & 0.24 & 0.04 & 1.00 & 1.16 & 1.32  \\\hline
  \multirow{3}{*}{6}   & 1 & 8.10 & 8.10 & 8.10 & 0.00 & 0.16 & 0.03 & 0.31 & 0.00 & 1.00 & 1.00 & 1.00  \\
                       & 3 & 10.25 & 10.25 & 10.25 & 0.01 & 0.32 & 0.04 & 0.35 & 0.01 & 1.00 & 1.00 & 1.00  \\
                       & 5 & 11.85 & 12.50 & 12.50 & 0.01 & 1.15 & 0.08 & 0.64 & 0.06 & 1.00 & 1.20 & 1.40  \\\hline
  \multirow{3}{*}{7}   & 1 & 9.00 & 9.00 & 8.47 & 0.01 & 0.87 & 14.69 & 1.68 & 0.00 & 1.00 & 1.00 & 1.00  \\
                       & 3 & 11.06 & 11.06 & 10.41 & 0.02 & 1.01 & 15.66 & 1.63 & 0.01 & 1.00 & 1.00 & 1.00  \\
                       & 5 & 13.47 & 13.47 & 12.71 & 0.03 & 2.11 & 33.70 & 2.55 & 0.13 & 1.00 & 1.18 & 1.82  \\\hline
\multirow{3}{*}{8}     & 1 & 10.08 & 10.08 & 10.08 & 0.01 & 0.81 & 0.02 & 1.19 & 0.00 & 1.00 & 1.00 & 1.00  \\
                       & 3 & 12.08 & 12.08 & 12.08 & 0.01 & 0.63 & 0.04 & 1.13 & 0.01 & 1.00 & 1.08 & 1.08  \\
                       & 5 & 14.23 & 14.23 & 14.23 & 0.01 & 1.55 & 0.06 & 3.53 & 0.02 & 1.00 & 1.08 & 1.08  \\\hline
\multirow{3}{*}{9}     & 1 & 11.00 & 11.00 & 11.00 & 0.03 & 3.36 & 0.04 & 10.86 & 0.00 & 1.00 & 1.00 & 1.00  \\
                       & 3 & 13.05 & 13.05 & 13.05 & 0.03 & 5.30 & 0.06 & 11.15 & 0.01 & 1.00 & 1.00 & 1.00  \\
                       & 5 & 15.10 & 15.10 & 15.10 & 0.03 & 4.32 & 0.09 & 8.43 & 0.02 & 1.00 & 1.00 & 1.00  \\\hline
\end{tabular}
\caption{Summary of the results of our experiments for the minimal mediated graph.\label{tab:minmg}}
\end{table}

In Figure \ref{fig:ppmin} we show the performance profiles (in log-scale to ease its reading) for our experiments on the minimal mediated graph distinguishing by dimension $d$, size of the set $\B$, $s$, and domain. Each of those lines were constructing by indicating the percent of instances optimally solved in less than a given number of seconds (in log-scale). The less time consuming approach (the superior lines) seems to be the domain $(2\Z)^d$ for $s=1$ (solid blue lines). In contrast, the more challenging problems were the continuous domain (green) and some of the $\Z^d$ domain (red). 

\begin{figure}[h!]
\includegraphics[width=\textwidth]{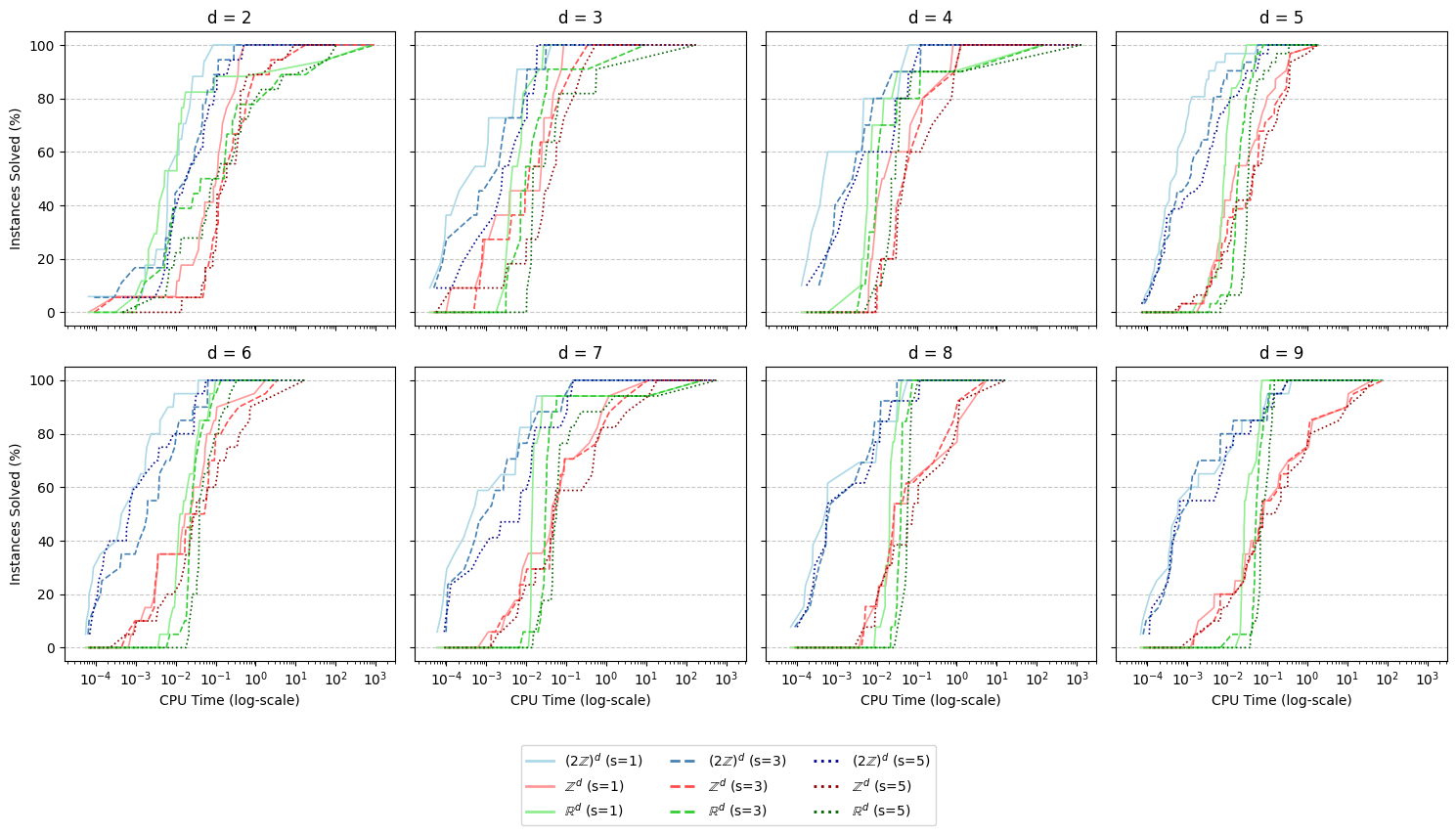}
\caption{Performance profile for our experiments on the minimal mediated graphs.\label{fig:ppmin}}
\end{figure}

In conclusion, our approach to compute minimal mediated graphs was capable to solve all the instances (even the $d=9$ ones) in more than reasonable CPU time, both to compute a single mediated graph or all the graphs in the family $\MinMG_A(\B)$, for all the domains that we consider. Thus, this novel approach with multiple application in conic decomposition has been proven to be useful and can be practically applied to these problems.

\section*{Conclusions}

In this paper we formally introduce the notion of mediated graph, motivated by the number of applications in convex geometry that have emerged in the last years on this type of structures. We give a first step by analyzing some of the graph theoretical properties of these graphs. We also focus on optimal mediated graphs, as those mediated graphs that are extremal for the partial order induced by the cardinality of their vertex sets. We address the computation of these minimal and maximal graphs on different domains by means of mixed integer linear optimization problems that allow to efficiently compute these subgraphs. The mathematical models that we propose are strengthened by using the theoretical properties that we study. An extensive computational experience allows to empirically validate our proposal, where we show that our framework allows to deal with reasonable sizes and dimensions of mediated graphs, being able to compute these structures in much less CPU time than previous proposals. For each of these structures we detail some applications in convex optimization where optimal mediated graphs have a direct impact. 

The structure of the optimal mediated graphs that we analyze here is still to be further analyzed. The existence of particular structures of sets $\A$ and $\B$ where the computation of minimal mediated graphs can be simplified is nowadays unknown. For maximal mediated graph, it was conjectured by \cite{reznick1989forms}, and empirically tested by  \cite{hartzer2022initial} for millions of simplices, that in dimension $2$ all simplices are either $H$-simplices or $M$-simplices (i.e., its maximal mediated graph is either the whose set of integer points inside the convex hull of $\A$ or it is the set $\A$ and its midpoints). The polyhedral study of the integer linear programs that we propose for its computation could give some light to prove the conjecture. On the other hand, in view of the applications of these structures in SOS-decomposition or SOC-representations, it may happen, for instance, that these representations are not possible or that are not efficient since they require large mediated graphs. In those cases, it is convinient to \emph{relax} the condition of the mediated graph to find \emph{close-enough} decomposition/stuctures. In this type of approximation schemes, the role of integer linear programming is still unknown.

\section*{Acknowledgements}

The authors were partially supported by grant PID2020-114594GB-C21 funded by MICIU/AEI/10.13039/501100011033;  grant RED2022-134149-T funded by MICIU/AEI
/10.13039/501100011033 (Thematic Network on Location Science and Related Problems); grant C-EXP-139-UGR23 funded by the Consejería de Universidad,
Investigación e Innovación and by the ERDF Andalusia Program 2021-2027, grant AT 21\_00032, and the IMAG-María de Maeztu grant CEX2020-001105-M /AEI /10.13039/501100011033. The second author was partially funded by LAAS CNRS LabEx CIMI (ANR-11-LABX-0040). 



\newpage
\appendix

\section{Proof of Theorem \ref{the:sobs}}\label{appendix}

\begin{proof}[Proof of Theorem \ref{the:sobs}]
    For the first part, notice that the polynomial $f$ can be written as a convex combination of the extremal nonnegative circuit with its support
    \begin{equation}
    f(\x)=\dsum_{m=0}^1\frac{\Theta_{f}+(-1)^mc}{2\Theta_{f}}\left[\dsum_{j=0}^dc_j\x^{\ba(j)}+(-1)^m\Theta_f\x^{\be}\right]
    \end{equation}

    Now, Note $\s^*_f$ is the solution of
\begin{equation}
\begin{bmatrix}
    \ba(1)-\ba(0)\\
    {\vv}dots\\
    \ba(d)-\ba(0)
\end{bmatrix}\s^*_f=\begin{bmatrix}
    \log\left(\frac{\lambda_1c_0}{\lambda_0c_1}\right)\\
    {\vv}dots\\
    \log\left(\frac{\lambda_dc_0}{\lambda_0c_d}\right)
\end{bmatrix}
\end{equation}
the linear system has a unique solution by affine independence of $\ba(0),\ba(1),\ldots,\ba(d)$. In this way, we have the extremal case $f^*_m(\x)=\dsum_{j=0}^dc_j\x^{\ba(j)}+(-1)^m\Theta_f\x^{\be}$, $m=0,1$, evaluated in $\left(e^{\s^*_f}\circ \x\right)$, where $\circ$ stands for the Hadamar product becomes 
\begin{equation}
    f^*_m\left(e^{\s^*_f}\circ \x\right)=\frac{c_0}{\lambda_0}e^{\langle \s^*_f,\ba(0)\rangle}\left[\dsum_{j=0}^d \lambda_j\x^{\ba(j)}+(-1)^m\x^{\be} \right].
\end{equation}

Now, $f^*_1\left(e^{\s^*_f}\circ \x\right)$ is a positive multiple of a simplicial agiform so by Reznick's SOS decomposition of agiforms \cite{reznick1989forms} and the fact of $\be\notin (2\Z)$ it easy to see that 

\begin{multline}
    f^*_m\left(e^{\s^*_f}\circ \x\right)=\\\frac{c_0}{\lambda_0}e^{\langle \s^*_f,\ba(0)\rangle}\left[c_{\be\be}\left(\x^{\frac{\be'}{2}} + (-1)^m\x^{\frac{\be''}{2}}\right)^2 + \dsum_{\ga\in V\setminus (\A\cup \{\be\})}c_{\be\ga}\left(\x^{\frac{\ga'}{2}}-\x^{\frac{\ga''}{2}}\right)^2\right].
\end{multline}
where $c_{\be\ga}\geq$ is the entry $(\be,\ga)$ of $(L^+(G)_{V\backslash \A})^{-1}$ for every $\ga\in V\backslash \A$ (Properties \ref{prop:prop}.\ref{prop:laplacian}). So, gathering the terms 

\begin{multline}
    f(e^{\s^*_f}\circ\x)=\frac{c_0}{\lambda_0}e^{\langle \s^*_f,\ba(0)\rangle}\dsum_{m=0}^1\frac{\Theta_{f}+(-1)^mc}{2\Theta_{f}}\left[c_{\be\be}\left(\x^{\frac{\be'}{2}} + (-1)^m\x^{\frac{\be''}{2}}\right)^2\right]\\
        + \frac{c_0}{\lambda_0}e^{\langle \s^*_f,\ba(0)\rangle}\dsum_{\ga\in V\setminus (\A\cup \{\be\})}c_{\be\ga}\left(\x^{\frac{\ga'}{2}}-\x^{\frac{\ga''}{2}}\right)^2,
\end{multline}
it just remains to undo the transformation to achieve the desired result.

To prove the second part, notice that the polynomial $f$ can be written as
    \begin{equation}
    f(\x)=\dsum_{j=0}^dc_j\x^{\ba(j)}-\Theta_f\x^{\be} + (\Theta_f+c)\x^{\be}=f^*_1(\x)+(\Theta_f+c)\left(\x^{\frac{\be}{2}}\right)^2.
    \end{equation}
    Now, use the same argument for $f^*_1(\x)$ in the proof of the first part and yield the result. 
\end{proof}

\end{document}